\documentclass[11pt]{article}%
\usepackage{amsfonts}
\usepackage{amsmath}
\usepackage{amssymb}
\usepackage{amsxtra}
\usepackage{graphicx}
\usepackage{geometry}
\usepackage{color}
\usepackage{colortbl}
\usepackage{caption}%
\providecommand{\U}[1]{\protect\rule{.1in}{.1in}}
\newtheorem{theorem}{Theorem}

\newtheorem{definition}[theorem]{Definition}

\newtheorem{lemma}[theorem]{Lemma}

\newtheorem{proposition}[theorem]{Proposition}
\newtheorem{remark}[theorem]{Remark}

\newcommand{\R} {\mathbb{R}}

\newcommand{\Z} {\mathbb{Z}}

\newcommand{\eps}{\epsilon}

\newcommand{\eto}{\stackrel{\eps\to 0}{\longrightarrow}}

\newcommand{\weto}{\stackrel{\eps\to 0}{\rightharpoonup}}

\newenvironment{proof}[1][Proof]{\noindent\textbf{#1.} }{\ \rule{0.5em}{0.5em}}
\numberwithin{equation}{section}
\ifx\pdfoutput\relax\let\pdfoutput=\undefined\fi
\newcount\msipdfoutput
\ifx\pdfoutput\undefined\else
\ifcase\pdfoutput\else
\ifx\paperwidth\undefined\else
\ifdim\paperheight=0pt\relax\else\pdfpageheight\paperheight\fi
\ifdim\paperwidth=0pt\relax\else\pdfpagewidth\paperwidth\fi
\fi\fi\fi
\begin{document}

\title{Asymptotic behavior of a Bingham Flow in thin domains \\ with rough boundary}
\author{G. Cardone\thanks{Universit\`{a} del Sannio, Dipartimento di Ingegneria, Corso
Garibaldi, 107, 82100 Benevento, Italy; email: gcardone@unisannio.it},
C. Perugia\thanks{Universit\`{a} del Sannio, Dipartimento di Scienze e Tecnologie, Via F. De Sanctis, Palazzo ex-Enel, 82100 Benevento,
Italy; email: cperugia@unisannio.it}, M. Villanueva Pesqueira\thanks{ Universidad Pontificia Comillas, Grupo Din\'amica No Lineal, Departamento Matem\'{a}tica Aplicada. ICAI Alberto Aguilera 25, 28015, Madrid, Spain; email:
mvillanueva@comillas.edu.}}
\maketitle

\begin{abstract}
\medskip We consider an incompressible Bingham flow in a thin domain with rough boundary, under the action of given external forces and with no-slip boundary condition on the whole boundary of the domain. In mathematical terms, this problem is described by non linear variational inequalities over domains where a small parameter $\eps$  denotes the thickness of the domain and the roughness periodicity of the boundary.  By using an adapted linear unfolding operator we perform a detailed analysis of the asymptotic behavior of the Bingham flow when $\eps$ tends to zero. We  obtain the homogenized limit problem for the velocity and the pressure, which preserves the nonlinear character of the flow, and study the effects of the microstructure in the corresponding effective equations. Finally, we give the interpretation of the limit problem in terms of a non linear Darcy law.
%\medskip In this paper, we analyze the steady flow of an incompressible Bingham flow in a thin domain with rough boundary, under the action of given external forces and with no-slip boundary condition on the whole boundary of the domain. Denoted  by $\eps$ the thickness of the domain and the roughness periodicity, this problem is described by non linear variational inequalities. We are interested in studying how the geometry of such a domain affects the asymptotic behaviour of the fluid when $\eps$ tends to zero. We  obtain the homogenized limit problem, which preserves the nonlinear character of the flow, and identify the effects of the microstructure in the corresponding effective equations. We conclude with the interpretation of the limit problem in terms of a non linear Darcy law.\\

\bigskip


\medskip


\end{abstract}

\section{Introduction}\label{intro}

In this paper we study the steady flow of an incompressible Bingham fluid in a thin domain with a rough boundary. Mathematical models involving thin domains are widely used to describe situations appearing naturally in numerous industrial and engineering applications. A relevant example is the classical lubrication problem, describing the relative motion of two adjacent surfaces separated by a thin film of fluid acting as a lubricant. In the incompressible case, the main unknown is the pressure of the fluid. Once resolved the pressure, it is possible to compute other fundamental quantities, such as the velocity field and the forces on the bounding surfaces.

On the other hand, to increase the hydrodynamic performance in various lubricated machine elements, for example journal bearings and thrust bearings, engineers also point out the importance of analyzing how the surface irregularities affects the thin film flow. From a mathematical point of view, a thin domain with rough boundary is usually described by two parameters $\eps$ and $\eta_\eps$, different in general, which tend to zero. The first one, $\eta_\eps$, is the characteristic wavelength of the periodic roughness, and $\eps$ is the thickness of the domain, i.e. the distance between the surfaces. There are several papers studying the asymptotic behavior of fluids in thin domains with rough boundary in the case of Newtonian fluids, see for instance \cite{BC,FKT,FTP} and the references therein. However, for the non-Newtonian fluids, the situation is completely different. The main reason is that the viscosity is a nonlinear function of the symmetrized gradient of the velocity (see \cite{AS}).

In this paper, denoted  by $\eps$ the thickness of the domain and the roughness periodicity, we are interested in studying how the geometry of the thin domain with rough boundary affects the asymptotic behaviur of an incompressible Bingham fluid, when $\eps$ tends to zero. 
We refer the reader to the very recent paper \cite{RK} and the references therein for the application of our study to problems issued from the real life applications. Indeed, predicting lava flow pathways is important for understanding effusive eruptions and for volcanic hazard assessment. One particular challenge is understanding the interplay between flow pathways and substrate topography that is often rough on a variety of scales ($< 1$ m to $10$ s km).

The Bingham fluid is a non-Newtonian fluid which behaves as a rigid body at low stresses but flows as a viscous fluid at high stress. This type of non-Newtonian fluid behavior is characterized by the existence of a threshold stress, called yield stress, which must be exceeded for the fluid to deform or flow. Once the externally applied stress is greater than the yield stress, the fluid exhibits Newtonian behavior. Typical examples of such fluids are some paints, toothpaste, the mud which can be used for the oil extraction, the volcanic lava or even the blood.

 The physical description of the Bingham fluid was introduced in \cite{B}, while the mathematical model of the Bingham flow in a bounded domain was performed by G. Duvaut and J.L. Lions in \cite{DL}. Here, the existence of the velocity and the pressure for such a flow was proved in the case of a bi-dimensional and of a three-dimensional domain.

 There are several papers studying the asymptotic behavior of Bingham fluids in thin domains. In particular we can mention \cite{BK, BK04}, where the asymptotic behavior of a Bingham fluid in a thin layer of thickness $\eps$ is studied. In \cite{BGL}, the authors obtain and analyze the limit problem for a steady incompressible flow of a Bingham fluid in a thin T-like shape structure. Finally, in the recent paper \cite{AB}, a dimension reduction and the unfolding operator method was used to describe the asymptotic behavior of the flow of a Bingham fluid in thin porous media. We also refer the reader to \cite{BM,BuCar, BuCarPe, LS}  and \cite{BuDo}, where the asymptotic behavior in porous media of a Bingham fluid and a power law fluid, respectively, is performed using different techniques in homogenization.

 Our paper is based on the periodic unfolding method, see \cite{CiorDamGri02}, for the first descriptions of the method, \cite{CiDamG, Cio-Dam-Don-Gri-Zaki}  for a systematic treatment of this method, and \cite{ArrVi, BG},  for an adaptation of this method to thin domains with oscillating boundaries. We refer to \cite{CiDamG2} for a further detailed description of the method, this book presents both the theory as well as numerous examples of applications. Thanks to this method, we are able to capture the microscopic behavior of the fluid near the rough boundary. Indeed, the unfolding operator allows us to obtain the homogenized limit although to establish suitable estimates for the pressure we need to adapt the extension operator introduced in \cite{BC}, generalizing the fundamental results of Tartar for porous domains \cite{Tar}, to the case of Bingham fluids.  

We underline that, following an approach similar to the one used to get our limit problem, we can recover the convergence results given in \cite{BC}, in the case of Newtonian fluids, see also \cite{FKT,FTP} for a generalization to the nonstationary case.

Let us point out that despite the works mentioned above, the study of a Bingham flow in a thin domain with a rough boundary has not been previously treated in literature.
 
 The paper is organized as follows. 
 
In Section 2, we introduce our thin domain with rough top boundary $\Omega_\eps$, where the parameter $\eps$ represents either the thickness of the domain or the rough periodicity.
Then, we formulate the problem which models the flow in $\Omega_\eps$ of a viscoplastic incompressible Bingham fluid with velocity $u_\eps$ and pressure $p_\eps$, verifying the nonlinear variational inequality \eqref{vfluid}. Finally, we give some notations useful in the sequel. 
In Section 3, we give some a priori estimates for both the velocity and the pressure. In Section 4, we introduce definition and properties of the unfolding operator, adapted to thin domains with oscillating boundary, introduced in \cite{ArrVi} for the bidimensional case. Section 5 is devoted to state some convergence results for the unfolded velocity field, taking into account the \emph{a priori} estimates proved in Section 3, a suitable "rescaled" velocity field, which is typical for this kind of problem in thin domains, and the unfolding operator defined in Section 4. 
Section 6  is dedicated to the extension of the pressure which is obtained assuming some restrictions on the domain. This extension have an essential importance in our study in order to get convergence results for the unfolded pressure. 
%We also establish an interesting relationship between the limit of the unfolded pressure and the pressure $p_\eps$ itself (see Proposition \ref{p0}). 
In Section 7, we state and prove the main result of our paper, Theorem \ref{teolim}, which allow us to identify the limit problem. Finally, in Section 8, we conclude with the interpretation of this limit problem, which preserves the nonlinear character of the flow. Indeed,  in the case of forces independent of the vertical variable, both a nonlinear Darcy equation and a lower dimensional Bingham-like law arise (see Proposition \ref{limitthin}). 

\begin{figure}\label{fig1}
\begin{center}
\includegraphics[scale=0.55]{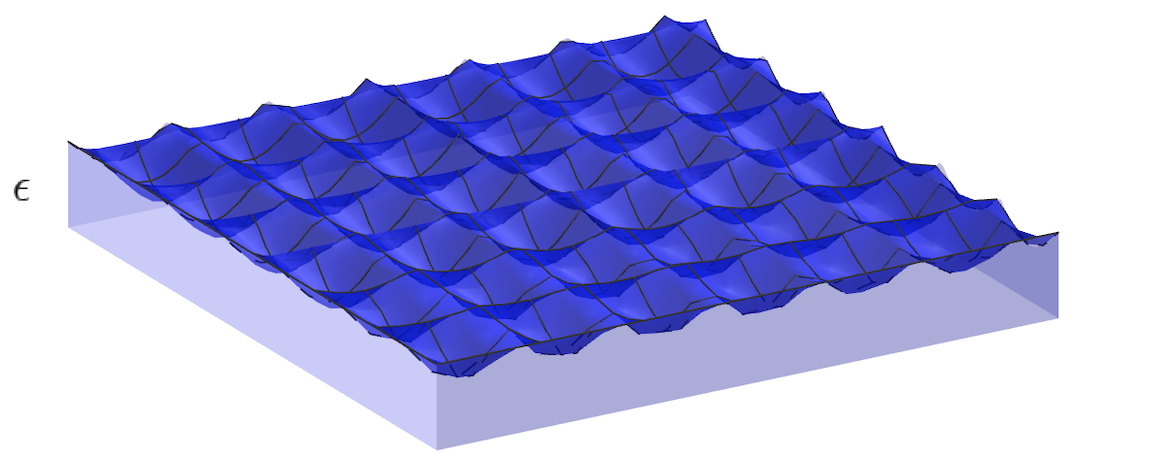}

\caption{Thin domain with oscillating periodic boundary}
\end{center}
\end{figure}

\section{The setting of the problem} \label{PRE}
Throughout the paper, we will consider three-dimensional thin domains with an oscillatory behavior in its top boundary, which are defined as follows (see Figure.1)
\begin{equation}\label{thin}
\Omega_\eps =  \Big\{ (x_1, x_2, x_3) \in \R^3 \; | \;  (x_1, x_2) \in \omega, \, 0<x_3<\eps G(x_1/\eps, x_2/\eps) \Big\},
\end{equation}
where $\omega=(0,1)^2$ denotes the unitary cell in $\R^2$, $\eps$ is a positive parameter tending to zero and $G: \R^2 \to \R$ is a smooth function, $Y$ - periodic, being  $Y=(0,L_1)\times(0,L_2)$ the periodicity cell, and   such that there exist two positive 
constants $G_0, G_1$ with $0< G_0\leq G(x_1,x_2)\leq G_1, \; \forall (x_1,x_2)\in \R^2$. 

In order to simplify the notation, we decompose each point $\bf x \in \R^3$ according to
$$ {\bf x}=(\hat{x},x_3), \hbox{ with } \hat{x}=(x_1,x_2) \in \R^2 \hbox{ and } x_3 \in \R.$$
Moreover, we also use the notation $\hat{\cdot}$ to denote a generic vector of $\R^2$.
Therefore, our thin domain is defined as follows
$$\Omega_\eps =  \Big\{ (\hat{x}, x_3) \in \R^3 \; | \;  \hat{x} \in \omega, \, 0<x_3<\eps G(\hat{x}/\eps) \Big\}.$$
The representative cell, which describes the thin structure, is given by
$$Y^*=\{y\equiv(\hat y, y_3) \in \R^3 \; | \;  \hat y \in Y,  \; 0 < y_3 < G(\hat y)\},$$ 
while its bottom and upper boundary will be denoted by $\partial Y^\ast_{inf}$ and $\partial Y^\ast_{sup}$, respectively.\\
Moreover, we define the domain with a fixed height $\Omega=\omega \times (0, G_1)$.

In $\Omega_\eps$, we consider the incompressible flow of a Bingham fluid, see \cite{B}, with viscosity and yield stress given by $\mu \eps^2$ and $g \eps$, respectively, where $\mu$ and $g$ are positive constants independent of $\eps$. The fluid velocity is denoted by $\bf{u}_\eps$, while the pressure of the fluid is denoted by $p_\eps$. Then, the stress tensor is defined by
\begin{equation}
\sigma_{ij}=-p_{\varepsilon
}\delta_{ij}+g\eps\dfrac{D_{ij}(\bf{u}_\eps)}{(D_{II}(\bf{u}_\eps%
))^{\frac{1}{2}}}+2\mu\eps^2 D_{ij}(\bf{u}_\eps),\label{sig}
\end{equation}
where $\delta_{ij}$ is the Kronecker symbol and $D_{ij}$ and $D_{II}$ are defined by
\begin{align*}
D_{ij}(\bf{u}_\eps) & =\frac{1}{2}\left(\frac{\partial u_{\eps,i}}{\partial x_{j}}+\frac{\partial u_{\eps,j}}{\partial
x_{i}}\right)  ,1\leq i,j\leq 3,\\
D_{II}(\bf{u}_\eps) &   =\dfrac{1}{2}%
%TCIMACRO{\dsum \limits_{i,j=1}^{n}}%
%BeginExpansion
{\displaystyle\sum\limits_{i,j=1}^{n}}
%EndExpansion
D_{ij}({\bf u}_\eps)D_{ij}(\bf{u}_\eps).
\end{align*}

\begin{remark}
Notice that we will denote vector fields in three dimensions using bold face, ${\bf u}_\eps=(u_{\eps,1},u_{\eps,2},u_{\eps,3})$. Moreover, the euclidean norm in $\R^3$ is denoted by $|\cdot|.$
\end{remark}

Relation (\ref{sig}) represents the constitutive law of the Bingham fluid. In \cite{DL}, it is shown that this constitutive law is equivalent to
the following one:%
\[
\left\{
\begin{array}
[c]{rcl}%
(\sigma_{II})^{\frac{1}{2}}<g\varepsilon & \Leftrightarrow & D_{ij}%
({\bf u}_{\varepsilon})=0\\
(\sigma_{II})^{\frac{1}{2}}\geq g\varepsilon & \Leftrightarrow &
D_{ij}({\bf u}_{\varepsilon})=\dfrac{1}{2\mu\eps^2}\left(  1-\dfrac{g\varepsilon}%
{(\sigma_{II}^{\varepsilon})^{\frac{1}{2}}}\right)  {\sigma}_{ij}^{\eps},
\end{array}
\right.
\]
where $\sigma_{II}$ and  $\sigma_{ij}^\eps$ are defined by
\begin{align*}
\sigma_{II}  &  =\dfrac{1}{2}%
%TCIMACRO{\dsum \limits_{i,j=1}^{n}}%
%BeginExpansion
{\displaystyle\sum\limits_{i,j=1}^{3}}
%EndExpansion
{\sigma}^{\eps}_{ij}{\sigma}_{ij}^{\eps},\\
\sigma_{ij}^\eps & = g \frac{D_{ij}}{(D_{II})^{\frac {1}{2}}} + 2 \mu \eps^2 D_{ij}.
\end{align*}

We assume that the fluid is incompressible, i. e. the velocity field is divergence free, 
and we impose the no-slip condition on the boundary of the domain, ${\bf u}_\eps=0$ on $\partial \Omega_\eps$. Therefore, the space of admissible velocity fields is given by
$$V_\eps= \Big\{ {\bf v} \in \big(H^1_0(\Omega_\eps)\big)^3 : \hbox{div}({\bf v})=0\Big\}.$$

Let us apply to the fluid an external body force ${\bf f}_\eps \in L^2(\Omega_\eps)^3$.\\

According to \cite{DL}, for any fixed $\epsilon$, the flow of our incompressible Bingham fluid is modeled by the following variational problem
\begin{equation}\label{vfluid}
\left\{
\begin{aligned}
\text{Find } {\bf u}_\eps \in V_\eps \text{ such that}\\
\mu \eps^2 \int_{\Omega_\eps} \nabla{\bf u}_\eps\cdot \nabla({\bf v}-{\bf u}_\eps)\, d{\bf x} &+ g\eps \int_{\Omega_\eps}|\nabla{\bf v}|\, d{\bf x} -g\eps \int_{\Omega_\eps}|\nabla{\bf u}_\eps|\, d{\bf x}\\
&\geqslant \int_{\Omega_\eps} {\bf f}_\eps \cdot ({\bf v}-{\bf u}_\eps)\, d{\bf x}, \; \forall {\bf v} \in V_\eps,
\end{aligned}
\right.
\end{equation}
which admits a unique solution in $V_\eps$.

Equivalently, see  \cite{BM, DL}, for any fixed $\epsilon$, denoted by  $p_\eps$ the pressure of the fluid in $\Omega_\epsilon$, there exists a unique couple $({\bf u}_\eps, p_\eps)\in V_\eps \times L^2_0(\Omega_\epsilon)$, satisfying the following variational inequality
\begin{equation}\label{vpfluid}
\begin{aligned}
\mu \eps^2 \int_{\Omega_\eps} \nabla{\bf u}_\eps\cdot \nabla({\bf v}-{\bf u}_\eps)\, d{\bf x} &+ g\eps \int_{\Omega_\eps}|\nabla{\bf v}|\, d{\bf x} -g\eps \int_{\Omega_\eps}|\nabla{\bf u}_\eps|\, d{\bf x}\\
&\geqslant \int_{\Omega_\eps} {\bf f}_\eps ({\bf v}-{\bf u}_\eps)\, d{\bf x} +  \int_{\Omega_\eps} p_\eps \hbox{div} ({\bf v}-{\bf u}_\eps)\, d{\bf x}, \; \forall {\bf v} \in H^1_0(\Omega_\eps)^3,
\end{aligned}
\end{equation}
where $L^2_0(\Omega_\epsilon)$ denotes the space of functions in $L^2(\Omega_\epsilon)$ with zero mean value.
\begin{remark}\label{rescaled norm}
Due to the order of the height of the thin domain, it makes sense to
 consider the following rescaled Lebesgue measure
$$\rho_{\eps}(\mathcal{O}) = \frac{1}{\eps} \mu(\mathcal{O}), \; \forall\, \mathcal{O} \subset \Omega_\eps,$$
which is widely considered in works involving thin domains, see e.g. \cite{HaRau,PerSil13,PriRi,Rau}.

As a matter of fact, from now on,  we use
the following rescaled norms in the thin open sets 
%Therefore, it makes sense to
% consider the following measure in the thin domain
%$$\rho_{\eps}(\mathcal{O}) = \frac{1}{\eps}|\mathcal{O}|, \; \forall\, \mathcal{O} \subset R^\eps.$$
%Observe that since the domains $R^\eps$ converge in some sense to the interval $(0,1)$ as $\eps$ tends to zero the rescaled Lebesgue measure $\rho_{\eps}$ allows
%us to preserve the relative capacity of a measurable subset $\mathcal{O} \subset R^\eps$.
%
%Moreover, using the previous measure we introduce the spaces $L^p( R^\eps, \rho_\eps)$ and  $W^{1,p}( R^\eps, \rho_\eps)$, for $1\leq p < \infty$
%endowed with the norms obtained rescaling the usual norms by the factor $\frac{1}{\eps}$, that is, 
\begin{align*}
|||\varphi|||_{L^p(\Omega_\eps)} &= \eps^{-1/p}||\varphi||_{L^p(\Omega_\eps)},  \quad \forall \varphi \in L^p(\Omega_\eps), \quad 1\leq p < \infty,\\
|||\varphi|||_{W^{1,p}( \Omega_\eps)} &= \eps^{-1/p}||\varphi||_{W^{1,p}( \Omega_\eps)}, \quad \forall \varphi \in W^{1,p}(\Omega_\eps), \quad 1\leq p < \infty.
\end{align*}
For completeness, we consider $|||\varphi|||_{L^\infty(\Omega_\eps)} =||\varphi||_{L^\infty(\Omega_\eps)}$ and we denote by $H^{-1}(\Omega_\eps)$ the dual space to $H^1_0( \Omega_\eps)$ endowed with the rescaled norm.

%for the dual space $H^{-1}(\Omega_\eps)$, the norm is given by
%\begin{equation}\label{defdual}
%|||\varphi|||_{H^{-1}(\Omega_\eps)}=\sup_{\psi \in H^1_0(\Omega_\eps)}\frac{1}{\eps}\frac{\left |\displaystyle {\int_{\Omega_\eps}}\varphi\psi d{\bf x}\right |}{|||\varphi|||_{H^1_0(\Omega_\eps)}}.
%\end{equation}
\end{remark}
\begin{remark}\label{body force}
Since the thin domain shrinks in the vertical direction as $\eps$ tends to zero, it is usual to assume that the applied forces do not depend on $\eps$ and they are of the form
$${\bf f}(x)=(\hat f(\hat x), 0), \hbox{ a.e. } x \in \Omega_\eps.$$
Notice that the third component is neglected and the force is independent of the vertical direction. Moreover, this particular ${\bf f}$ satisfies
$$\eps^{-1/2}|| {\bf f} ||_{L^2(\Omega_\epsilon)^3}\leq C || {\hat f} ||_{L^2(\omega)^2}\leq C.$$
\end{remark}
Throughout the paper, we suppose 
\begin{equation}\label{Hf}
||| {\bf f}_\eps |||_{L^2(\Omega_\epsilon)^3} \leq C,
\end{equation} 
for some positive constant C independent of $\eps$.\\
This assumption on the applied forces is usual in order to obtain appropriate estimates. In fact, as already observed in Remark \ref{body force}, the common choice of the applied forces ${\bf f}_\eps$ in thin domains, where the forces do not depend on the vertical variable and the vertical component of the forces is neglected, satisfies this assumption.

\section{\emph {A priori} estimates}\label{APEst}

In this section, we follow the standard procedure to get the  \emph {a priori} estimates for the velocity $\bf{u}_\eps$ and  the pressure $p_\eps$.

First, notice that, the Poincar\'e inequality in the thin domain \eqref{thin} can be written as
\begin{equation}\label{Pi}
||| \boldsymbol{\varphi}|||_{L^2(\Omega_\eps)^{3}}\leq \eps C|||\nabla\boldsymbol{\varphi}|||_{L^2(\Omega_\eps)^{3\times3}}, \quad \forall \boldsymbol{\varphi} \in H^1_0(\Omega_\eps)^3,
\end{equation}
where $C$ is independent of $\boldsymbol{\varphi}$ and $\eps$.

\begin{lemma}
For any fixed $\eps$, let $({\bf u}_\eps, p_\eps)$ be the solution of \eqref{vpfluid}. Under the assumption \eqref{Hf}, the following estimates hold
\begin{align}
&\eps|||\nabla {\bf u}_\eps|||_{L^2(\Omega_\eps)^{3\times3}}\leq C,\label{v}\\
 & ||| {\bf u}_\eps|||_{L^2(\Omega_\eps)^{3}}\leq C,\label{gradv} \\
 & |||\nabla p_\eps|||_{H^{-1}(\Omega_\eps)^{3}}\leq \eps C,\label{p}
\end{align}
with $C$ a positive constant independent of $\eps$.
\end{lemma}

\begin{proof}
Taking ${\bf v}=0$ and ${\bf v}=2\bf{u}_\eps$ as a test function in \eqref{vpfluid} we get
\begin{equation*}
\begin{aligned}
-\mu \eps^2 \int_{\Omega_\eps} \nabla{\bf u}_\eps \cdot \nabla{\bf u}_\eps\, d{\bf x} &-g\eps \int_{\Omega_\eps}|\nabla{\bf u}_\eps|\, d{\bf x}\geqslant - \int_{\Omega_\eps} {\bf f}_\eps {\bf u}_\eps\, d{\bf x},\\
\mu \eps^2 \int_{\Omega_\eps} \nabla{\bf u}_\eps \cdot \nabla{\bf u}_\eps\, d{\bf x} &+g\eps \int_{\Omega_\eps}|\nabla{\bf u}_\eps|\, d{\bf x}\geqslant  \int_{\Omega_\eps} {\bf f}_\eps {\bf u}_\eps\, d{\bf x}.\\
\end{aligned}
\end{equation*}
Consequently, we obtain
$$\mu \eps^2 \int_{\Omega_\eps} \nabla{\bf u}_\eps \cdot \nabla{\bf u}_\eps\, d{\bf x} +g\eps \int_{\Omega_\eps}|\nabla{\bf u}_\eps|\, d{\bf x}= \int_{\Omega_\eps} {\bf f}_\eps {\bf u}_\eps\, d{\bf x}.$$ 
By using H\"{o}lder's inequality on the right hand side and the assumption \eqref{Hf}, we have
$$\mu \eps^2  |||\nabla{\bf u}_\eps|||^2_{L^2(\Omega_\eps)^{3\times3}}\leq ||| {\bf f}_\eps|||_{L^2(\Omega_\epsilon)^3} ||| {\bf u}_\eps|||_{L^2(\Omega_\eps)^{3}}\leq C ||| {\bf u}_\eps|||_{L^2(\Omega_\eps)^{3}},$$
with $C$ a positive constant independent of $\eps$.

Then, applying the  Poincar\'e inequality \eqref{Pi}, we obtain

$$\mu \eps^2  |||\nabla{\bf u}_\eps|||^2_{L^2(\Omega_\eps)^{3\times3}}\leq \eps C |||\nabla{\bf u}_\eps|||_{L^2(\Omega_\eps)^{3\times3}}.$$  
Therefore, from this last inequality and by \eqref{Pi} again, we get estimates \eqref{v} and \eqref{gradv}.

Finally, we are going to obtain the \emph{a priori} estimate for the pressure. To this aim, let ${\bf v_\eps} \in \big(H^1_0(\Omega_\eps)\big)^3$. Then, taking ${\bf v}={\bf v}_\eps+{\bf u}_\eps$ as a test function in \eqref{vpfluid}, we get
\begin{equation*}
\begin{aligned}
\mu \eps^2 \int_{\Omega_\eps} \nabla{\bf u}_\eps\cdot \nabla{\bf v_\eps}\, d{\bf x} &+ g\eps \int_{\Omega_\eps}|\nabla{\bf v_\eps}+\nabla{\bf u_\eps} |\, d{\bf x} -g\eps \int_{\Omega_\eps}|\nabla{\bf u}_\eps|\, d{\bf x}\\
&\geqslant \int_{\Omega_\eps} {\bf f}_\eps \cdot {\bf v_\eps}\, d{\bf x} +  \int_{\Omega_\eps} p_\eps \hbox{div}({\bf v}_\eps)\, d{\bf x}, \; \forall {\bf v}_\eps \in H^1_0(\Omega_\eps)^3.
\end{aligned}
\end{equation*}
Hence, by Holder's inequality, it follows that
$$\frac{1}{\eps} \int_{\Omega_\eps} p_\eps \hbox{div}({\bf v}_\eps)d{\bf x}\leq \mu \eps^2 |||\nabla{\bf u}_\eps|||_{L^2(\Omega_\eps)^{3\times3}}|||\nabla{\bf v}_\eps|||_{L^2(\Omega_\eps)^{3\times3}} + g\eps|||\nabla{\bf v}_\eps|||_{L^2(\Omega_\eps)^{3\times3}}+ ||| {\bf f}_\eps |||_{L^2(\Omega_\epsilon)^3} ||| {\bf v}_\eps|||_{L^2(\Omega_\eps)^{3}}.$$
Consequently, by using \eqref{Pi} and estimates \eqref{Hf} and  \eqref{v}, we get
$$\frac{1}{\eps} \int_{\Omega_\eps} p_\eps \hbox{div}({\bf v}_\eps)d{\bf x}\leq C\eps|||\nabla{\bf v}_\eps|||_{L^2(\Omega_\eps)^{3\times3}}, \quad  \forall {\bf v}_\eps \in H^1_0(\Omega_\eps)^3,$$
which provides estimate \eqref{p}.
\end{proof}

\section{The unfolding operator}\label{UO}
In this section, we extend, to three dimensional thin domains with an oscillatory boundary, the definition of the unfolding operator, which was introduced in \cite{ArrVi} in the two dimensional case. Moreover, we present  some of the main properties of the unfolding operator which we will need in order to obtain the homogenized limit problem. 

We will use similar notations as in \cite{ArrVi}  :
\begin{itemize}
\item $N_\eps$ denotes the largest integer such that $\eps L_1(N_\eps+1)\leqslant1$,
\item $M_\eps$ denotes the largest integer such that $\eps L_2(M_\eps+1)\leqslant1$,
\item $\omega_{ij}^\eps= (i\eps L_1, (i+1)\eps L_1) \times  (j\eps L_2, (j+1)\eps L_2)$ with $i=0,1,\cdots,N_\eps,\,j=0,1,\cdots,M_\eps$,
\item $\omega^\eps= \hbox{Int}\left\{\displaystyle \bigcup_{i=0}^{N_\eps}\displaystyle \bigcup_{j=0}^{M_\eps} \overline{ \omega_{ij}^\eps}\right\}$, $\overline{ \omega_{ij}^\eps}$ denotes the closure of the open set $\omega_{ij}^\eps$,
\item $\Lambda^\eps= \omega \setminus \omega^\eps$ or equivalently, $\Lambda^\eps = \big([\eps L_1(N_\eps +1), 1)\times (0,1)\big) \cup \big((0,1) \times [\eps L_2(M_\eps +1), 1)\big)$,
\item $\Omega^0_\eps$ denotes the set which contains all the cells totally included in $\Omega_\eps$ 
$$\Omega^0_\eps = \Big\{ ( \hat x, x_3) \in \R^3 \; | \;  \hat x \in \omega^\eps,  \; 0 < x_3 < \epsilon \, G(\hat x/\eps) \Big\},$$
\item $\Omega^1_\eps=\Omega_\eps \setminus \Omega^0_\eps.$\\

\item By analogy with the definition of the integer and fractional part of a real number, for $\hat x\in \R^2$, $[\hat x]_{L}$ denotes the unique pair of integers, $[\hat x]_{L}=(k_1,k_2)\in \Z^2$, such that $\hat x \in  \big[ k_1 L_1, (k_1 + 1) L_1\big) \times \big[ k_2 L_2, (k_2 + 1) L_2\big)$ and $\{\hat x\}_L \in [0, L_1)\times [0,L_2)$ is such that $\hat x = [\hat x]_{L}L + \{\hat x\}_L$. 
Then, if $L$ denotes the pair $(L_1,L_2)$, for each $\eps>0$ and for every $\hat x\in \R^2$, there exists a unique pair of integers, $\Big[\frac{\hat x}{\eps}\Big]_L$, such that 
\begin{equation}\label{decomposition}
\hat x = \eps\Big[\frac{\hat x}{\eps}\Big]_LL +  \eps\Big\{\frac{\hat x}{\eps}\Big\}_L, \quad  \Big\{\frac{\hat x}{\eps}\Big\}_L \in [0, L_1)\times [0,L_2).
\end{equation}
\end{itemize}

We are now in position to define the unfolding operator in our setting.

\begin{definition}\label{unfold def}
Let $\varphi$ be a Lebesgue-measurable function defined in $\Omega_\eps$. The unfolding operator $\mathcal{T_\eps}$, acting on $\varphi$, is defined as the following function in $\omega \times Y^*$
\begin{eqnarray*}
\mathcal{T}_\eps(\varphi) (\hat x, \bf{y})
= 
 \left\{ 
\begin{array}{ll}
\varphi \Big( \eps \Big[\frac{\hat x}{\eps}\Big]_{L}L + \eps \hat y,\eps y_3\Big) &\hbox{for} \quad (\hat x , \hat y, y_3) \in \omega^\eps \times Y^*, \\
0  &\hbox{for} \quad  (\hat x , \hat y, y_3) \in  \Lambda^\eps \times Y^*.
\end{array}
\right.
\end{eqnarray*}
\end{definition}

In the following proposition, we list the main properties of the unfolding operator previously defined. 
\begin{proposition}\label{properties}
The unfolding operator $\mathcal{T_\eps}$ has the following properties:
\begin{itemize}
\item[i)] $ \mathcal{T_\eps}$ is a linear operator.
\item[ii)] $\mathcal{T}_\eps(\varphi \psi)=\mathcal{T_\eps(\varphi)} \mathcal{T_\eps(\psi)}$ $\, \forall \, \varphi, \psi$ Lebesgue-measurable functions in $\Omega_\eps$.
\item[iii)] Let $\varphi \in L^1(\Omega_\eps).$ The following integral equality holds
\begin{align*}
& \frac{1}{L_1L_2}\int_{ \omega \times Y^*} \mathcal{T_\eps(\varphi)} (\hat x, {\bf y})\, d{\hat x} d{\bf y} = \frac{1}{\eps}\int_{\Omega^0_\eps} \varphi ({\bf x})\, d{\bf x}\\
&= \frac{1}{\eps}\int_{\Omega_\eps} \varphi ({\bf x}) \,d{\bf x} - \frac{1}{\eps}\int_{\Omega^1_\eps} \varphi ({\bf x})\, d{\bf x}.
\end{align*}
\item[iv)] For every $\varphi \in L^p(\Omega_\eps)$, we have $\mathcal{T_\eps(\varphi)} \in L^p\big( \omega \times Y^*\big)$, with $1\leq p \leq \infty$. In addition, the following relationship exists between their norms:
$$\|\mathcal{T_\eps(\varphi)}\|_{L^p\big( \omega \times Y^*\big)} = (L_1L_2)^{\frac{1}{p}}  \, |||\varphi|||_{L^p(\Omega^0_\eps)} \leq (L_1L_2)^{\frac{1}{p}}  \, |||\varphi|||_{L^p(\Omega_\eps)}.$$ 
In the special case $p=\infty$, 
$\|\mathcal{T_\eps(\varphi)}\|_{L^\infty \big( \omega\times Y^*\big)}= \|\varphi\|_{L^\infty(\Omega_\eps^0)} \leq \|\varphi\|_{L^\infty(\Omega^\eps)}.$
\item[v)] For every $\varphi \in W^{1,p}(\Omega_\eps)$, $1\leq p \leq \infty$, one has
\begin{equation}\label{derivative}
 \frac{\partial}{\partial y_i}\mathcal{T_\eps(\varphi)} = \eps \mathcal{T_\eps}\Big(\frac{\partial \varphi}{\partial x_i}\Big), \quad  \hbox{ for } i=1,2,3.
 \end{equation}
\item[vi)] Let $\varphi$ be a measurable function on $Y^*$, extended by $Y-$periodicity in the first two variables. Then $\varphi^\eps(\hat{x},x_3) = \varphi (\frac{\hat{x}}{\eps},\frac{x_3}{\eps})$ is 
a measurable function on $\Omega^\eps$, such that
$$\mathcal{T}_\eps(\varphi^\eps)(\hat x,{\bf y}) = \varphi({\bf y}), \quad \forall (\hat x, {\bf y}) \in \omega^\eps \times Y^*.$$
Furthermore, if $\varphi \in L^p(Y^*)$, with $1\leq p \leq \infty$ then  $\varphi^\eps \in L^p(\Omega_\eps)$.
\item[vii)] Let $\{\varphi^\eps\}$ be a sequence of functions in $L^p(\omega)$, $1\leq p < \infty$, such that
 $$ \varphi^\eps \eto \varphi \quad \hbox{strongly in } L^p(w).$$
 Then
 $$\mathcal{T_\eps(\varphi^\eps)} \eto  \varphi \quad \hbox{strongly in } L^p\big( \omega \times Y^*\big).$$
%\item[viii)] For every ${\bf v} \in H^1(\Omega_\eps)^3$ we have 
%\begin{equation}\label{nonlinear}
%\eps \mathcal{T_\eps}(|\nabla{\bf v}|)=|\nabla_y \mathcal{T_\eps}({\bf v})|.
%\end{equation}
\end{itemize}
\end{proposition}
\begin{remark}
The proofs of these properties are omitted since they follow directly from the properties proved in \cite{ArrVi} for the bidimensional case. 
Notice that, in view of property iii) in Proposition \ref{properties}, we may say that the unfolding operator ``almost preserves'' the integral of the functions, since the ``integration defect'' arises only from the cells which are not completely included in $\Omega_\eps$ and it is controlled by the integral on $\Omega_\eps^1$.
% Therefore, we assume from now on that the domain $\Omega_\eps$ is given by an exact union of basic cells, that is, $\Omega_\eps^1=\emptyset$. Therefore, the unfolding operator preserves the integrals.
%Moreover, property {\emph viii)} was proved in \cite{BuCarPe}.
\end{remark}

%\begin{remark}
%Notice that, in view of property \emph{iii)}, we may say that the unfolding operator ``almost preserves'' the integral of the functions since the ``integration defect'' arises only from the cells which are not completely included in $\Omega_\eps$ and it is controlled by the integral on $\Omega_\eps^1$.
% Therefore, we assume from now on that the domain $\Omega_\eps$ is given by an exact union of basic cells, that is, $\Omega_\eps^1=\emptyset$. Therefore, the unfolding operator preserves the integrals.
% \end{remark} 

For every vector field ${\bf v} \in H^1(\Omega_\eps)^3$ the unfolding operator is naturally defined as follows:
\begin{equation}\label{Tvect}
\mathcal{T_\eps}({\bf v})=(\mathcal{T_\eps}({v_1}),\mathcal{T_\eps}({v_2}),\mathcal{T_\eps}({v_3})).
\end{equation}
Therefore, using basic properties of the unfolding operator, we prove the following proposition.
\begin{proposition}\label{propgradv}
 For every ${\bf v} \in H^1(\Omega_\eps)^3$ we have 
\begin{equation}\label{nonlinear}
\eps \mathcal{T_\eps}(|\nabla{\bf v}|)=|\nabla_{\bf y} \mathcal{T_\eps}({\bf v})|.
\end{equation}
\end{proposition}
{\bf Proof.} By ii) and v) of Proposition \ref{properties} and by \eqref{Tvect}, we get
\begin{align*}
\bigl[T_{\varepsilon}(|\nabla {\bf v}|)\bigr]^{2} &  =T_{\varepsilon}(|\nabla {\bf v}|^{2})=T_{\varepsilon
}\Bigl(\sum_{i,j=1}^{3}\Bigl (\frac{\partial v_{i}}{\partial x_{j}}%
\Bigr)^{2}\Bigr)=\sum_{i,j=1}^{3}\Bigl(T_{\varepsilon}\Bigl(\frac{\partial v_{i}}{\partial
x_{j}}\Bigr)\Bigr)^{2}\\
&  =\sum_{i,j=1}^{3}\Bigl(\frac{1}{\varepsilon}%
\frac{\partial}{\partial y_{j}}T_{\varepsilon}(v_{i})\Bigr)^{2}=\frac
{1}{\varepsilon^{2}}|\nabla_{\bf y}T_{\varepsilon}({\bf v})|^{2},
\end{align*}
which is \eqref{nonlinear}.
$\blacksquare$

\section{Some convergence results for the velocity}\label{ConvV}

In this section, we state some weak convergences for the velocity field, taking into account the \emph{a priori} estimates \eqref{v} and \eqref{gradv}.

In order to analyze the asymptotic behavior of the velocity field, we first perform a simple and typical change of variables in thin domains, which consists
in stretching in the $x_3$-direction by a factor $1/\eps$, i. e. $y_3=x_3/\eps$. Then, the thin domain $\Omega_\eps$ is transformed into the domain
$$\tilde \Omega_\eps=\Big\{ (\hat{x}, y_3) \in \R^3 \; | \;  \hat{x} \in \omega, \, 0<y_3< G(\hat{x}/\eps) \Big\}.$$
Notice that the rescaled domain $\tilde \Omega_\eps$ is not thin anymore, although it still presents an oscillatory behavior on the upper boundary. 

Then, we introduce the rescaled velocity field through the following notations:
\begin{align*}
& {\bf U}_\eps (\hat x, y_3)= {\bf u}_\eps(\hat x, \eps y_3), \; \hbox{ a.e. } (\hat x, y_3) \in \tilde \Omega_\eps,\\
&\big(\nabla_\eps  {\bf U}_\eps \big)_{i,j}= \partial_{x_j}  {U }_\eps^i, \;  \big(\nabla_\eps {\bf U}_\eps \big)_{i,3}= \frac{1}{\eps}\partial_{y_3}  {U }_\eps^i, \hbox{ for } i=1,2,3, \, j=1,2,\\
&\hbox{div}_\eps  {\bf U}_\eps = \partial_{x_1}  { U }^1_\eps+ \partial_{x_2}  {U }^2_\eps +\frac{1}{\eps}\partial_{y_3}  {U }^3_\eps.
\end{align*}

Let $\Omega=\omega \times (0, G_1)$ be the rectangular parallelepiped introduced in Section \ref{PRE}.
Since the domain $\tilde \Omega_\eps$ ``converges'' in some sense to  $\Omega$, as is usual in classical homogenization, extension of ${\bf U }$ to the whole $\Omega$ can be used to obtain suitable estimates in the fixed domain $\Omega$ and to pass to the limit.

%give the limit for the rescaled velocity field ${\bf U}_\eps$ and we show some of the properties of the limit.

\begin{proposition}\label{behavior velocity}
Let $\tilde {\bf U }_\eps \in H^1_0(\Omega)^3$ be the extension by zero of ${\bf U}_\eps$ to $\Omega$. Then, up to a subsequence, still denoted by $\eps$, there exists ${\bf U} \in H^1((0,G_1);L^2(\omega)^3)$ such that
\begin{equation}\label{conv1}
\tilde {\bf U }_\eps \weto  {\bf U} \quad \hbox{w}-H^1((0,G_1);L^2(\omega)^3).
\end{equation}
Moreover,  ${\bf U } = (\hat U, 0)$ satisfies
\begin{equation}\label{conditionsU}
\left\{
\begin{aligned}
\mathrm{div}_{\hat x}\left(\int_{0}^{G_1}{\hat U}(\hat x, y_3)\, dy_3 \right) =0 &\hbox{ in } \omega,\\
\left(\int_{0}^{G_1}{\hat U}(\hat x, y_3) \, dy_3 \right) \cdot n =0 &\hbox{ on }  \partial \omega,
\end{aligned}
\right.
\end{equation}
where $n$ is the outward normal to $\omega$.
\end{proposition}

\begin{proof}
From the \emph{a priori} estimates \eqref{v} and \eqref{gradv} we deduce
$$\left\Vert\tilde {\bf U }_\eps\right\Vert_{L^2(\Omega)^3}\leq C,\; \left\Vert\frac{\partial \tilde {\bf U }_\eps}{\partial y_3} \right\Vert_{L^2(\Omega)^3}\leq C, \; \eps \left\Vert\frac{\partial \tilde {\bf U }_\eps}{\partial x_i} \right\Vert_{L^2(\Omega)^3}\leq C, \quad i=1,2. $$
Therefore, there exists ${\bf U} \in H^1((0,G_1);L^2(\omega)^3)$ such that, up to a subsequence, we have
\begin{align*}
\tilde {\bf U }_\eps \weto {\bf U} &\quad \hbox{ w}-L^2(\Omega)^3,\\
\frac{\partial \tilde {\bf U }_\eps}{\partial y_3} \weto \frac{\partial \bf U}{\partial y_3}  &\quad \hbox{ w}-L^2(\Omega)^3,\\
\eps\frac{\partial \tilde {\bf U }_\eps}{\partial x_i} \weto {\bf z}_i  &\quad \hbox{ w}-L^2(\Omega)^3, \; i=1,2.
\end{align*}
Moreover, taking into account that $\dfrac{\partial \tilde {\bf U }_\eps}{\partial x_i}$ is bounded in $H^{-1}(\Omega)^3$, we get $ {\bf z}_i=0$, for $i=1,2.$

Now, we are going to prove that $U_3=0$. The incompressibility condition implies that\\
%\textcolor{red}{Components don't need to be written in bold that's why I changed}
\begin{equation}\label{incom}
\eps\frac{\partial \tilde {U }^1_\eps}{\partial x_1}+\eps\frac{\partial \tilde {U} ^2_\eps}{\partial x_2}+ \frac{\partial \tilde {U}^3 _\eps}{\partial y_3}=0.
\end{equation}
Consequently, we have
$$\int_{\Omega } \left(\eps\frac{\partial \tilde {U }^1_\eps}{\partial x_1}+\eps\frac{\partial \tilde {U} ^2_\eps}{\partial x_2}+ \frac{\partial \tilde {U}^3 _\eps}{\partial y_3}\right) \varphi \; dx=0, \; \forall \varphi \in \mathcal{D}(\Omega).$$
Passing to the limit  we get
\begin{equation}\label{x3}
 \int_{\Omega }\frac{\partial {U^3 }}{\partial y_3} \varphi \; dx=0, \; \forall \varphi \in \mathcal{D}(\Omega),
 \end{equation}
which implies that $U^3$ does not depend on $y_3$.  

On the other side, the continuity of the trace operator from the space of functions $v$ such that $\|\tilde{v}\|_{L^2(\Omega)}$ and $\|\partial_{y_3}\tilde{v}\|_{L^2(\Omega)}$ are bounded to $L^2(\omega\times \{G_1\})$ and to $L^2(\omega\times \{0\})$ implies 
%$\tilde{v}=0$ on $\omega\times \{G_1\}a$ and $\omega\times \{0\}$ and the boundary conditions on ${\bf %U }_\eps$ give that 
\begin{equation}\label{boundary}
{\bf U}(\hat x, 0)={\bf U}(\hat x, G_1)=0.
\end{equation}
Hence, combining \eqref{x3} and \eqref{boundary}, we prove that $ {U^3 }=0$.

In order to prove \eqref{conditionsU}, let $\varphi \in \mathcal{D}(\omega)$. Multiplying \eqref{incom} by $\frac{1}{\eps}\varphi$ and integrating by parts, we get
$$\int_{\Omega} \left(\tilde{U}^1_\eps \frac{\partial \varphi}{\partial x_1} + \tilde{U}^2_\eps \frac{\partial \varphi}{\partial x_2}\right)\, d{\hat x}dy_3=0.$$
Passing to the limit, by \eqref{conv1}, we get the result. \end{proof}
\bigskip

Now, we should take into account that the extension by zero of the velocity does not capture the effects of the rough boundary.
Therefore, in the next proposition, we get the limit ${\bf u}$ for the unfolded velocity field $\mathcal{T}_\eps({\bf u}_\eps)$, which  helps us to understand how the microscopic geometry of the domain affects the behavior of the fluid. Moreover, we show the relationship between ${\bf u}$ and ${\bf U}$.

\begin{proposition}\label{limit u}
Let ${\bf u}_\eps$ be the solution of \eqref{vfluid}.  Then, up to a subsequence, still denoted by $\eps$, there exists ${\bf u} \in L^2(\omega;H^1(Y^*)^3)$ such that
\begin{align}
\mathcal{T}_\eps({\bf u}_\eps)\weto  {\bf u} &\quad \hbox{w}-L^2\big( \omega; H^1(Y^*)^3\big),\label{uconv1}\\
 \eps \mathcal{T_\eps}\Big(\frac{\partial {\bf u}_\eps}{\partial x_i}\Big)\weto  \frac{\partial {\bf u}}{\partial y_i} ,&\quad \hbox{w}-L^2\big(\omega\times Y^*\big)^3, i=1,2,3,\label{uconv2}\\
\mathrm{div}_{\bf y}{\bf u}=0 &\quad \hbox{ in } \omega \times Y^*,\label{div}\\
{\bf u}=0 &\quad \hbox{ on } \omega \times \partial Y^\ast_{sup}\cup \omega \times \partial Y^\ast_{sup}\label{bound}
%\{y_3=G(\hat{y})\}%. \label{bound}
\end{align}

Moreover, since the function ${\bf u}$ satisfies the following conditions
%\begin{align}
%&\mathrm{div}_{y}{\bf u}=0 \hbox{ in } \omega \times Y^*,\label{div}\\
%&{\bf u}=0 \hbox{ on } \omega \times \{y_3=0\}\cup \omega \times \{y_3=G(\hat{y})\}. \label{bound}
%\end{align}
%%\textcolor{red}{Probably it is better to write $\{y_3=0\}$ in place of $\partial_{inf}Y^*$ and $\{y_3=G(\hat{y})\}$ in place of $\partial_{sup} Y^*$}\\
%and since $\displaystyle {\bf U }=\frac{1}{L^2}\int_{Y}\tilde{\bf u }(\hat x,\hat y, y_3)d\hat y$ it satisfies
%\textcolor{blue}{I don't see the mistake in the proof of this equality, in fact, I think both equalities are true. Taking into account that $\tilde{\bf u }$ denotes the extension by zero of  $u$ if we integrate in the previous equality respect to $y_3$ we get the same equality as you wrote in red.
%$$\int_{0}^ {G_{1}} {\bf U} \,dy_3=\frac{1}{L^2}\int_{Y^*}{\bf u }(\hat x,\hat y, y_3)dy.$$
%However, perhaps we can use the same relation as in Suarez because, as you say, properties 6.12 and 6.13 are obvious from the equality that you have written.}
%
%\textcolor{red}{I think it is wrong. The corrected one is
%\begin{equation}\label{defU}
%\int_{0}^ {G_{1}} {\bf U} \,dy_3=\frac{1}{L^2}\int_{Y^*}{\bf u }(\hat x,\hat y, y_3)d\hat y
%\end{equation}
%See comments in the proof}
\begin{align}
\int_{Y\ast} u_3\, d\bf y=0, \label{cond0bis}\\
\mathrm{div}_{\hat x}\left(\int_{Y^*}{ \hat u} d\bf{y} \right) =0 &\quad \hbox{in } \omega,\label{cond1}\\
\left(\int_{Y^*} {\hat u}\, d{\bf y} \right) \cdot n =0 &\quad \hbox{ on } \partial\omega.\label{cond2}
\end{align}

\end{proposition}

\begin{proof}
From the \emph{a priori} estimates \eqref{v}, \eqref{gradv} and taking into account property \emph{iv)} in Proposition \ref{properties}, we have
$$\|\mathcal{T_\eps}({\bf u}_\eps)\|_{L^2( \omega \times Y^*)^3} \leq C, \quad  \left\Vert \eps \mathcal{T_\eps}\Big(\frac{\partial {\bf u}_\eps}{\partial x_i}\Big)\right\Vert_{L^2\big( \omega \times Y^*\big)^3} \leq C,\quad \; i=1,2,3.$$

Therefore, in view of property \emph{v)} in Proposition \ref{properties}, we can ensure the existence of ${\bf u} \in L^2(\omega; H^1(Y^*)^3)$, such that convergences \ref{uconv1} and \ref{uconv2} hold, up to a subsequence.

Moreover, since $ {\bf u}_\eps \in V_\eps$, by $i)$ of Proposition \ref{properties}, we have 
$$\sum_{i=1}^3 \mathcal{T_\eps}\Big(\frac{\partial {\bf u}_\eps}{\partial x_i}\Big)=0.$$
Then, multiplying the above equality by $\eps$, using \eqref{derivative} and passing to the limit, we easily obtain \eqref{div}.

Finally, by using the $Y$-periodicity of the function G, and taking into account that $ {\bf u}_\eps$ is zero on the boundary of $\Omega_\eps$, we get
\begin{equation*}
\begin{aligned}
\mathcal{T}_\eps( {\bf u}_\eps)_{|  \omega \times \partial Y^\ast_{inf}}&=\mathcal{T}_\eps( {\bf u}_\eps)_{|y_3=0}=\mathcal{T}_\eps ( {{\bf u}_\eps}_{|x_3=0})=0,\\
\mathcal{T}_\eps( {\bf u}_\eps)_{| \omega \times \partial Y^\ast_{sup}}&= \mathcal{T}_\eps( {\bf u}_\eps)(\hat x, \hat y, G(\hat y))=  {\bf u}_\eps \Big( \eps \Big[\frac{\hat x}{\eps}\Big]_{L}L + \eps \hat y,\eps G(\hat y)\Big)\\
&= {\bf u}_\eps \Big(\eps \Big[\frac{\hat x}{\eps}\Big]_{L}L + \eps \hat y,\eps G(\Big[\frac{\hat x}{\eps}\Big]_{L}L + \hat y)\Big)=0,
\end{aligned}
\end{equation*}
which implies \eqref{bound} on the trace of $\bf u$.

In order to prove \eqref{cond1} and \eqref{cond2}, we need to establish the relation between $\bf u$ and the limit of the rescaled velocity $\bf U$, defined in Proposition \ref{behavior velocity}. To this aim, let us consider $\boldsymbol\varphi \in \mathcal{D}(\omega)^3$. Then, by using the definitions of the rescaled operator and the unfolding operator, we have
\begin{align*}
&\int_{\Omega}\tilde {\bf U }_\eps {\boldsymbol \varphi}\; d{\hat x}dy_3=\int_{\tilde \Omega_\eps} {\bf U }_\eps \boldsymbol\varphi\; d{\hat x}dy_3=\frac{1}{\eps}\int_{\Omega_\eps} {\bf u }_\eps(x_1,x_2,x_3) \boldsymbol\varphi(x_1,x_2)\; d{\bf x}\\
&= \frac{1}{L_1 L_2}\int_{ \omega \times Y^*} \mathcal{T_\eps({\bf u }_\eps)}\mathcal{T}_\eps(\boldsymbol\varphi) d{\hat x} d{\bf y}.
%+\frac{1}{\eps}\int_{\Omega^1_\eps} {\bf u }_\eps(x_1,x_2,x_3) \varphi(x_1,x_2,x_3/\eps)\; dx .
\end{align*}
By convergences \eqref{conv1} and \eqref{uconv1}, we can pass to the limit on the left and right hand side and obtain
$$\int_{\Omega}{\bf U } \boldsymbol\varphi\; d{\hat x}dy_3= \frac{1}{L_1 L_2}\int_{ \omega \times Y^*} {\bf u } \boldsymbol\varphi\;d{\hat x} d{\bf y}, \quad \forall \boldsymbol\varphi \in \mathcal{D}(\omega)^3.$$
%\textcolor{red}{why $\varphi\Big( \eps \Big[\frac{\hat x}{\eps}\Big]_{L}L + \eps \hat y,y_3\Big)$ should converge to $\varphi$?}\\
%\textcolor{blue}{The proof of this convergence is the following:
% \begin{equation*}
% \begin{split}
% \| \varphi\Big( \eps \Big[\frac{\hat x}{\eps}\Big]_{L}L + \eps \hat y,y_3\Big)- \varphi(\hat x, y_3)\|^p_{L^p\big( \omega \times Y^*\big)} & \leq \int_{\omega \times Y^*} |m_{\varphi}(\eps)|^p \;d{\hat x} d{\bf y}  
% \end{split}
% \end{equation*}
% where $m_{\varphi}(\eps)$ is the modulus of continuity of the function $\varphi$
%  $$m_{\varphi}(\eps) = \sup_{|x - y|<\eps}|\varphi(x) - \varphi(y)|.$$
% So, since $\varphi$ is uniformly continuous on $\Omega$  we get the following strong convergence
% \begin{equation*}
% \mathcal{T_\eps(\varphi^\eps)}  \eto \varphi, \qquad \forall \varphi \in \mathcal{D}(\Omega), \, \varphi^\eps=\varphi(x_1,x_2,x_3/\eps).
%\end{equation*} 
%}

Consequently, we get
$$\int_{\omega}\left(\int_0^{G_{1}} {\bf U }(\hat x, y_3)dy_3 \right)\boldsymbol\varphi(\hat x)\; d{\hat x}= \frac{1}{L_1L_2}\int_{\omega}\left(\int_{Y^\ast}{\bf u} (\hat x,\hat y, y_3)d\hat y dy_3\right)\boldsymbol\varphi(\hat x) \, d{\hat x}\qquad \forall \boldsymbol\varphi \in \mathcal{D}(\omega)^3,$$
which naturally implies
\begin{equation}\label{cond0}
\int_{0}^ {G_{1}} {\bf U} (\hat x)\,dy_3=\frac{1}{L_1 L_2}\int_{Y^*}{\bf u }(\hat x,{\bf y})\, d{\bf y}, \text{ for a.e. }\hat x\in \omega.
\end{equation}

% 
% \textcolor{red}{I think the following one is the corrected proof in analogy with what happens in Suarez\\
% To do this, we consider $\varphi \in \mathcal{D}(\omega)^3$. Then, using the definition of the rescaled operator and the unfolding operator we have
%\begin{align*}
%&\int_{\Omega}\tilde {\bf U }_\eps \varphi\; d{\hat x}dy_3=\int_{\tilde \Omega_\eps} {\bf U }_\eps \varphi\; d{\hat x}dy_3=\frac{1}{\eps}\int_{\Omega_\eps} {\bf u }_\eps(x_1,x_2,x_3) \varphi(x_1,x_2)\; d{\bf x}\\
%&= \frac{1}{L^2}\int_{ \omega \times Y^*} \mathcal{T_\eps({\bf u }_\eps)}\mathcal{T_\eps({\bf \varphi }_\eps)} d{\hat x} d{\bf y}.
%%+\frac{1}{\eps}\int_{\Omega^1_\eps} {\bf u }_\eps(x_1,x_2,x_3) \varphi(x_1,x_2,x_3/\eps)\; dx .
%\end{align*}
%Taking into account convergence \eqref{uconv1} and $vii$ in Proposition \ref{properties}, we can pass to the limit on the left and right hand side. Thus, we obtain
%$$\int_{\Omega}{\bf U } \varphi\; d{\hat x}dy_3= \frac{1}{L^2}\int_{ \omega \times Y^*} {\bf u } \varphi \;d{\hat x} d{\bf y}$$
%i. e. \ref{defU}.}

 Moreover, since $U^3=0$, by \eqref{cond0} we have \eqref{cond0bis}. Finally, \eqref{conditionsU} and \eqref{cond0} immediately imply \eqref{cond1} and \eqref{cond2}.
% \textcolor{red}{From your definition of $U$ I can't deduce  \eqref{cond1} and \eqref{cond2} as you said. From my \ref{defU} it is possible...}
\end{proof}

\section{Convergence results for the pressure}\label{ConvP}
Obtaining appropriate convergences for the pressure is not immediate. Notice that, by the \emph{a priori} estimate \eqref{p} and by Ne$\breve{c}$as inequality, we have 
$$||| p_\eps|||_{L^2(\Omega_\eps)}\leq C(\Omega_\eps)|||\nabla p_\eps|||_{H^{-1}(\Omega_\eps)^{3}}\leq \eps C(\Omega_\eps).$$

Therefore, it is not obvious how to obtain an estimate of the pressure, in order to get a convergence result. To overcome this difficulty, in previous papers, the extension operator introduced by Tartar in \cite{Tar} was used. In this sense, in 
\cite{BC,Mi} a generalization of the results of Tartar was introduced for the case of Newtonian fluids in a thin film flow with a rough boundary, and in \cite{AS} the authors perform a generalization to the case of a non-Newtonian fluid governed by the Navier-Stokes system.
In this paper, we extend the previous results to the case of a Bingham fluid.  

We consider a smooth surface included in the basic cell $Y^*$ and surrounding the hump such that $Y^*$ is split into two regions $Y_f$ and $Y_m$, (see Figure 2).

We denote:
\begin{align*}
&W=Y\times (0, G_1),\\
&Y_s=W \setminus(Y_m \cup Y_f),\\
& S= \partial Y_m \cap \partial Y_f,
\end{align*}
while the upper boundary of $U$ will be denoted by $\Gamma$.

We suppose from now on the following assumptions:
\begin{itemize}
\item{H1)}  the surface roughness is made of detached smooth humps periodically given on the upper part of the gap,
\item{H2)} the thin domain is given by an exact number of basic cells, that is  $\Lambda_\eps=\emptyset$,
\item{H3)} $\partial Y_m$ is a $C^1$ manifold.
\end{itemize}

\begin{figure}\label{fig2}
\begin{center}
\includegraphics[scale=0.5]{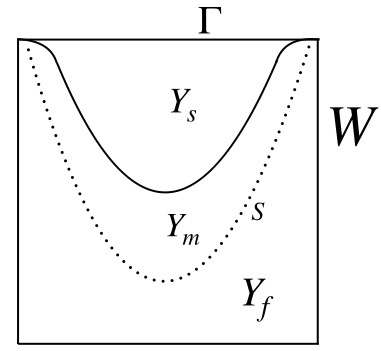}
\caption{Representative cell}
\end{center}
\end{figure}
Therefore, the following lemma holds, (see Lemma 3.1 in \cite{BC}).

\begin{lemma}\label{ext}
Let $\boldsymbol \varphi$ be a function in $H^1(W)^3$ such that $\boldsymbol \varphi=0$ on $\Gamma$. Then, there exists $\boldsymbol \psi \in H^1(Y_m)^3$ such that:
$$\boldsymbol \psi_{|_ S}=\boldsymbol \varphi_{|_ S} \quad \hbox{ and } \quad \boldsymbol \psi_{|_ {\partial Y_m \setminus S}}=0.$$
Moreover, there exists a constant $C$, not depending on $\boldsymbol \varphi$, such that
\begin{align*}
\|\boldsymbol \psi\|_{H^1(Y_m)^3}\leq C \|\boldsymbol \varphi\|_{H^1(W)^3}
\end{align*}
and $\text{div}_\eps \boldsymbol \varphi=0$ implies $\text{div}_{\eps} \boldsymbol \psi=0$.
\end{lemma}

The previous lemma allows us to construct a restriction operator from the rectangle
$Q_\eps=\omega \times (0,\eps G_1)$ to the thin domain $\Omega_\eps$, (see Lemma 3.2 in \cite{BC}).

\begin{lemma}\label{restriction}
There exists an operator
$R^\eps: H^1_0(Q_\eps)^3 \to H^1_0(\Omega_\eps)^3$ such that
\begin{enumerate}
\item for $\boldsymbol \varphi \in H^1_0(\Omega_\eps)^3$, $R^\eps(\boldsymbol \varphi)=\boldsymbol \varphi$;
\item $\hbox{div} \boldsymbol \varphi=0$ implies $\hbox{div} R^\eps(\boldsymbol \varphi)=0$;
\item for any function $\boldsymbol \varphi \in H^1_0(Q_\eps)^3$ there exists a constant $C$, independent of $\boldsymbol \varphi$ and $\eps$, such that
\begin{align*}
&||| R^\eps(\boldsymbol \varphi)|||_{L^2(\Omega_\eps)^3}\leq C\big( |||\boldsymbol  \varphi |||_{L^2(Q_\eps)^3}+ \eps|||\nabla \boldsymbol \varphi|||_{L^2(Q_\eps)^{3\times3}} \big)\\
& |||\nabla R^\eps(\boldsymbol \varphi)|||_{L^2(\Omega_\eps)^{3\times3}}\leq C\big( \frac{1}{\eps}||| \boldsymbol \varphi |||_{L^2(Q_\eps)^3}+|||\nabla \boldsymbol \varphi|||_{L^2(Q_\eps)^{3\times3}} \big)\\
\end{align*}
\end{enumerate} 
\end{lemma}

\begin{proof}
For the reader's convenience, following the same idea of \cite{BC}, Lemma 3.1 and \cite{AS}, Lemma 4.6, we will give an indication on how to obtain this restriction operator. Notice that, for any $\boldsymbol \varphi \in H^1(W)^3$ such that $\boldsymbol \varphi=0$ on $\Gamma$, Lemma \ref{ext} allows us to define $R(\boldsymbol \varphi)$ by
\begin{equation*}
R(\boldsymbol \varphi)({\bf y})=
\left\{
\begin{aligned}
\boldsymbol \varphi({\bf y}) &\hbox{ if } {\bf y} \in Y_f,\\
\psi({\bf y}) &\hbox{ if }  {\bf y} \in Y_m,\\
0 &\hbox{ if }  {\bf y} \in Y_s,\\
\end{aligned}
\right.
\end{equation*}
which satisfies $||R(\boldsymbol \varphi)||_{H^1(W)^3}\leq C||\boldsymbol \varphi||_{H^1(W)^3}.$
Then, by assumption H2), we can define $R^\eps$ by applying $R$ to each cell.
\end{proof}
\medskip

By using the extension operator, obtained by duality argument from $R^\eps$, we obtain the required estimate, (see \cite{AS, BC} for details), and an important convergence result as stated by the following proposition.

\begin{proposition}\label{limitp}
Let $({\bf u}_\eps,p_\eps)$ be the solution of \eqref{vpfluid}.  Then, there exists an extension $P_\eps$ of $p_\eps$ to  $Q_\eps$ such that
\begin{equation}\label{estimate pressure}
|||P_\eps |||_{L^2(Q_\eps)}\leq C,
\end{equation}
with $C$ a  positive constant independent of $\eps$.
Moreover, up to a subsequence, still denoted by $\eps$,  there exists $p \in L^2(\omega \times Y^*)$, independent of ${\bf y}$, such that
\begin{equation}\label{weakp}
\mathcal{T}_\eps(P_\eps|_{\Omega_\eps})\weto  p \quad \hbox{w}-L^2(\omega \times Y^*).
\end{equation}

\end{proposition}
\begin{proof} The proof developes into two steps.\\

\textbf{Step 1.} Let us construct the extension $P_\eps$ to $Q_\eps$ of the pressure $p_\eps$ and prove estimate \eqref{estimate pressure}.
To this aim, let us observe that the operator $R^\eps$ defined in Lemma \ref{restriction} allows us to extend the pressure $p_\eps$ to $Q_\eps$ introducing $F_\eps$ in $H^{-1}(Q_\eps)^3$ defined as follows
\begin{equation}\label{idP}
\langle F_\eps, \boldsymbol  \varphi\rangle_{Q_\eps}=\langle \nabla p_\eps, R^\eps(\boldsymbol \varphi) \rangle_{\Omega_\eps}, \quad \forall \boldsymbol \varphi \in H^1_0(Q_\eps)^3.
\end{equation} 
Now, let us estimate the right hand side by using the variational inequality \eqref{vpfluid}. To this aim, let us take successively ${\bf v}={\bf u}_\eps+R^\eps(\boldsymbol \varphi)$ and ${\bf v}={\bf u}_\eps-R^\eps(\boldsymbol \varphi)$ in \eqref{vpfluid} and have

\begin{equation}\label{F}
|\langle F_\eps, \boldsymbol \varphi \rangle_{Q_\eps}| \leq \mu \eps^2 \left | \int_{\Omega_\eps} \nabla{\bf u}_\eps\cdot \nabla(R^\eps ( \boldsymbol \varphi))\, d{\bf x}\right | +  g\eps \int_{\Omega_\eps}|\nabla (R^\eps ( \boldsymbol \varphi))|\, d{\bf x} +\left|\int_{\Omega_\eps} {\bf f}_\eps R^\eps (\boldsymbol \varphi)\, d{\bf x}\right|.
\end{equation}

\noindent Moreover, by \textit{2.} in Lemma \ref{restriction} and  identification \eqref{idP}, $\hbox{div} {\boldsymbol \varphi}=0$ implies 
$$\langle F_\eps, \boldsymbol \varphi \rangle_{Q_\eps}=0.$$
Hence, the DeRham theorem gives the existence of $P_\eps$ in $L^2_0(Q_\eps)$ such that $F_\eps=\nabla P_\eps$ and, by Holder inequality,  \eqref{F} implies

\begin{align*}
\frac{1}{\eps}|\langle \nabla P_\eps, \boldsymbol \varphi  \rangle_{Q_\eps}| \leq &\mu \eps^2 |||\nabla{\bf u}_\eps|||_{L^2(\Omega_\eps)^{3\times3}}|||\nabla (R^\eps (\boldsymbol \varphi)) |||_{L^2(\Omega_\eps)^{3\times3}} \\
&+ g\eps|||\nabla (R^\eps (\boldsymbol \varphi)) |||_{L^2(\Omega_\eps)^{3\times3}}+ ||| {\bf f}_\eps |||_{L^2(\Omega_\epsilon)^3} ||| R^\eps (\boldsymbol \varphi )|||_{L^2(\Omega_\eps)^{3}}.
\end{align*}
Consequently, by using \eqref{Pi} and estimates \eqref{Hf} and \eqref{v}, we get

$$\frac{1}{\eps}|\langle \nabla P_\eps, \boldsymbol \varphi \rangle_{Q_\eps}|\leq C\eps|||\nabla(R^\eps (\boldsymbol \varphi))|||_{L^2(\Omega_\eps)^{3\times3}},  \; \forall \boldsymbol \varphi \in H^1_0(Q_\eps)^3.$$
Then, by \textit{3.} in Lemma \ref{restriction}, we have
\begin{equation}\label{estpi}
\frac{1}{\eps}|\langle \nabla P_\eps, \boldsymbol \varphi \rangle_{Q_\eps}|\leq C\big(||| \boldsymbol\varphi |||_{L^2(Q_\eps)^3}+\eps |||\nabla \boldsymbol \varphi|||_{L^2(Q_\eps)^{3\times3}} \big).
\end{equation}

\noindent Using the dilatation 
$$y_3=\dfrac{x_3}{\eps}$$
as in Section \ref{ConvV}, let us set $\pi_\eps(x,y_3)=P_\eps(x, \eps y_3)$ and $\nabla_\eps \pi_\eps=(\nabla_{\hat{x}} P_\eps, \frac{1}{\eps}\partial_{y_3} P_\eps)$. Hence, for any $ \boldsymbol \Phi \in H_0^1(\Omega)^3$, we get
\begin{equation}\label{eqpi}
\langle \nabla_{\eps} \pi_\eps, {\boldsymbol \Phi} \rangle_{\Omega}=-\int_{\Omega}\pi_\eps\hbox{div}_\eps\boldsymbol \Phi\,d{\bf x}=-\dfrac{1}{\eps}\int_{Q_\eps}P_\eps\hbox{div}\boldsymbol \varphi\,d{\bf x}=\dfrac{1}{\eps}\langle \nabla P_\eps, {\boldsymbol \varphi} \rangle_{Q_\eps},
\end{equation}
where $ \boldsymbol  \varphi (\hat x, x_3)=\boldsymbol \Phi(\hat{x},\frac{x_3}{\eps})$. Finally, taking into account that $||| \boldsymbol\varphi |||_{L^2(Q_\eps)^3}=||\boldsymbol \Phi||_{L^2(\Omega)^3}$ and $||| \nabla \boldsymbol\varphi |||_{L^2(Q_\eps)^{3\times 3}}=||\nabla_\eps \boldsymbol \Phi||_{L^2(\Omega)^{3\times 3}}$ , by \eqref{estpi} and \eqref{eqpi}, we can write
$$|\langle \nabla_{\eps} \pi_\eps, {\boldsymbol \Phi} \rangle_{\Omega}|=\dfrac{1}{\eps}|\langle \nabla_{\eps} \pi_\eps, {\boldsymbol \varphi} \rangle_{Q_\eps}|\leq C\big(|| \boldsymbol\Phi ||_{L^2(\Omega)^3}+\eps ||\nabla \boldsymbol \Phi||_{L^2(\Omega)^{3\times 3}} \big)\leq C || \boldsymbol \Phi ||_{H^1_0(\Omega)^{3}},$$
which implies
$$||\nabla_\eps \pi_\eps||_{H^{-1}(\Omega)^{3}} \leq C.$$
It follows that (see for instance \cite{GR}, Chapter I, Corollary 2.1) there exists a representative of $\pi_\eps\in L^2_0(\Omega)$ such that
$$
 ||\pi_\eps||_{L^2_0(\Omega)} \leq ||\nabla \pi_\eps||_{H^{-1}(\Omega)^{3}} \leq ||\nabla_\eps \pi_\eps||_{H^{-1}(\Omega)^{3}} \leq C,
$$
which easily implies \eqref{estimate pressure}. \\
Moreover observe that if $\boldsymbol\varphi\in H^1_0(\Omega_\eps)^3$, its extension by zero to $Q_\eps$, denoted by $\widetilde{\boldsymbol\varphi}$, is in $H^1_0(Q_\eps)^3$ and $R^\eps(\widetilde{\boldsymbol \varphi})=\boldsymbol \varphi$. Hence, taking into account \eqref{idP}, by integrating we obtain
\begin{equation}\label{pipressure}
-\int_{\Omega_\eps}P_\eps|_{\Omega_\eps} \hbox{div}\boldsymbol \varphi\, d{\bf x}=-\int_{Q_\eps}P_\eps \hbox{div}\widetilde{\boldsymbol \varphi}\, d{\bf x}=\langle \nabla P_\eps,\widetilde{\boldsymbol \varphi}\rangle_{Q_\eps}=\langle \nabla p_\eps, R^\eps(\boldsymbol \varphi) \rangle_{\Omega_\eps}=-\int_{\Omega_\eps}p_\eps \hbox{div}\boldsymbol\varphi\, d{\bf x}.
\end{equation}
Thus, for each $\boldsymbol\varphi\in H^1_0(\Omega_\eps)^3$ we get 
$$\int_{\Omega_\eps}(P_\eps|_{\Omega_\eps}-p_\eps) \hbox{div}\boldsymbol \varphi\, d{\bf x}=0,$$
which implies $P_\eps=p_\eps$ in $L^2(\Omega_\eps)\slash\R.$\\

\medskip

\textbf{Step 2.} Now let us prove convergence \eqref{weakp} and the independence of ${\bf y}$ of function $p$.\\
To this aim let us observe that by \eqref{estimate pressure} and $iv)$ of Proposition \ref{properties} we get
$$|| \mathcal{T}_\eps(P_\eps|_{\Omega_\eps})||_{L^2(\omega \times Y^*)} \leq C |||P_\eps|_{\Omega_\eps}|||_{L^2(\Omega_\eps)}\leq C |||P_\eps|||_{L^2(Q_\eps)}\leq C$$
which implies, by weak compactness, convergence \eqref{weakp}.

Finally, we shall prove that $p$ does not depend on $\bf y$. To this aim, let us consider ${\bf v}^\eps({\bf x})=\phi(\hat x){\boldsymbol \psi}(\bf x/\eps)$ where $\phi \in \mathcal{D}(\omega)$ and ${\boldsymbol \psi }\in H^1(Y^*)^3$ such that $\boldsymbol \psi$ is $Y-$periodic and ${\boldsymbol \psi}=0 \hbox{ on }\partial Y^\ast_{inf}\cup \partial Y^\ast_{sup}$. Of course ${\bf v}^\eps \in H^1_0(\Omega_\eps)$ and, in view of $vi)$ in Proposition \ref{properties}, it satisfies
\begin{align}
&\frac{\partial {\bf v}^\eps}{\partial x_i} =  \frac{\partial\phi}{\partial x_i}(\hat x)\boldsymbol \psi\Big(\frac{\bf x}{\eps}\Big)+ \frac{1}{\eps} \phi(\hat x) \frac{\partial\boldsymbol \psi}{\partial y_i}\Big(\frac{\bf x}{\eps}\Big),\, i=1,2, \quad \frac{\partial {\bf v}^\eps}{\partial x_3}= \frac{1}{\eps}\phi(\hat x) \frac{\partial\boldsymbol \psi}{\partial y_3}\Big(\frac{\bf x}{\eps}\Big),\label{gradtest}\\
&\mathrm{div}_{\bf x}({\bf v^\eps})=\hat\psi\Big(\frac{\bf x}{\eps}\Big) \mathrm{div}_{\hat x}\phi(\hat{x})+\frac{1}{\eps} \mathrm{div}_{\bf y}\boldsymbol \psi\Big(\frac{\bf x}{\eps}\Big).\nonumber
\end{align}
Hence, by using properties \emph{v)}, \emph{vi)} and \emph{vii)} in Proposition \ref{properties}, we get
\begin{align}
\mathcal{T_\eps}({\bf v}^\eps) \eto \phi \boldsymbol \psi &\quad \hbox{s-}L^2(\omega \times Y^*)^3,\label{testconv1}\\
\eps \mathcal{T_\eps}\Big(\frac{\partial {\bf v}^\eps}{\partial x_i}\Big) \eto  \phi \frac{\partial \boldsymbol \psi}{\partial y_i} &\quad \hbox{s-}L^2(\omega \times Y^*)^3, \; i=1,2,\label{testconv2}\\
\eps\mathcal{T_\eps}\Big(\frac{\partial {\bf v}^\eps}{\partial x_3}\Big) \eto  \phi \frac{\partial \boldsymbol \psi}{\partial y_3} &\quad \hbox{s-}L^2(\omega \times Y^*)^3,\label{testconv3}\\
\eps\mathcal{T_\eps}\Big(\mathrm{div}_{\bf x}{\bf v}_\eps\Big) \eto  \mathrm{div}_{\bf y}\boldsymbol \psi &\quad \hbox{s-}L^2(\omega \times Y^*)^3 .\label{testconv4}
\end{align}
Now let us take ${\bf u}_\eps+\eps{\bf v}_\eps$ as test function in \eqref{vpfluid}. We have
$$ \eps\int_{\Omega_\eps} p_\eps {\rm div}_{\bf x}{\bf v}_\eps\, d{\bf x}\leq \mu \eps^3 \int_{\Omega_\eps} \nabla {\bf u}_\eps\cdot \nabla {\bf v}_\eps\, d{\bf x}+g\eps^2 \int_{\Omega_\eps}|\nabla{\bf v}_\eps|\, d{\bf x}-\eps \int_{\Omega_\eps} {\bf f}_\eps {\bf v}_\eps \, d{\bf x}$$
and by \eqref{pipressure}
$$ \eps\int_{\Omega_\eps}P_\eps|_{\Omega_\eps} {\rm div}_{\bf x}{\bf v}_\eps\, d{\bf x}\leq \mu \eps^3 \int_{\Omega_\eps} \nabla {\bf u}_\eps\cdot \nabla {\bf v}_\eps\, d{\bf x}+g\eps^2 \int_{\Omega_\eps}|\nabla{\bf v}_\eps|\, d{\bf x}-\eps \int_{\Omega_\eps} {\bf f}_\eps {\bf v}_\eps \, d{\bf x}.$$
Then, we apply the unfolding operator to the previous inequality. By property $ii)$ and \emph{iii)} in Proposition \ref{properties}, we have 
\begin{equation}\label{unfoldvar}
\begin{aligned}
&\eps\int_{\omega\times Y^*} \mathcal{T}_\eps(P_\eps|_{\Omega_\eps})\mathcal{T_\eps}\Big(\mathrm{div}_{\bf x}{\bf v}_\eps\Big)\, d{\hat x}d{\bf y}\leq \mu \eps^3 \int_{\omega\times Y^*} \mathcal{T}_\eps(\nabla{\bf u}_\eps)\cdot \mathcal{T}_\eps(\nabla{\bf v}_\eps)\, d{\hat x}d{\bf y} + g\eps^2 \int_{\omega\times Y^*}\mathcal{T}_\eps(|\nabla{\bf v}_\eps)|\, d{\hat x}d{\bf y} \\
&-\eps \int_{\omega\times Y^*}\mathcal{T}_\eps ({\bf f}_\eps)\mathcal{T}_\eps({\bf v}_\eps)\, d{\hat x}d{\bf y}.
\end{aligned}
\end{equation} 
According to convergences \eqref{uconv2}, \eqref{testconv2}, \eqref{testconv3}, we get 
\begin{equation}\label{pres1}
\mu \eps^2 \int_{\Omega_\eps} \mathcal{T}_\eps( \nabla {\bf u}_\eps) \cdot \mathcal{T}_\eps(\nabla {\bf v}_\eps)\, d{\bf x}\to \mu  \int_{\omega\times Y^*}\nabla_{\bf y} u \cdot \big(\phi \nabla_{\bf y}\boldsymbol \psi)\, d{\hat x}d{\bf y}.
\end{equation}
Hence, the first integral in the right-hand side of \eqref{unfoldvar} satisfies
\begin{equation}\label{pres1bis}
\mu \eps^3 \int_{\Omega_\eps}\mathcal{T}_\eps( \nabla {\bf u}_\eps)\cdot \mathcal{T}_\eps(\nabla {\bf v}_\eps)\, d{\bf x}\to 0.
\end{equation}
By Proposition \ref{propgradv}, one has the following equality
$$g\eps \int_{\omega\times Y^*}\mathcal{T}_\eps(|\nabla{\bf v^\eps})|\, d{\hat x}d{\bf y} = g \int_{\omega\times Y^*} |\nabla_{\bf y} \mathcal{T}_\eps({\bf v^\eps})|\, d{\hat x}d{\bf y}.$$
Notice that
\begin{align*}
&\Big|\int_{\omega\times Y^*} |\nabla_{\bf y} \mathcal{T}_\eps({\bf v^\eps})|\, d{\hat x}d{\bf y} -  \int_{\omega\times Y^*} |\phi \nabla_{\bf y} \boldsymbol \psi|\, d{\hat x}d{\bf y}\Big|
\leq \int_{\omega\times Y^*} \big|\nabla_{\bf y}\mathcal{T}_\eps({\bf v^\eps})-\phi \nabla_{\bf y}\boldsymbol \psi\big|\, d{\hat x}d{\bf y}\\
& = \int_{\omega\times Y^*}  \big| \eps \mathcal{T}_\eps(\nabla \phi)\boldsymbol \psi+\mathcal{T}_\eps(\phi)\nabla_{\bf y}\boldsymbol \psi-\phi \nabla_{\bf y} \boldsymbol \psi\big|\, d{\hat x}d{\bf y}.
\end{align*}
By using properties \emph{vii)} in Proposition \ref{properties}, we have
\begin{align*}
\eps \mathcal{T}_\eps(\nabla \phi)\eto 0 &\quad \hbox{s-}L^2(\omega \times Y^*),\\
\mathcal{T}_\eps(\phi)\eto\phi &\quad \hbox{s-}L^2(\omega \times Y^*),
\end{align*}
then, we get the following convergence
\begin{equation}\label{pres2}
g\eps \int_{\omega\times Y^*}\mathcal{T}_\eps(|\nabla{\bf v^\eps})|\, d{\hat x}d{\bf y} \to g\int_{\omega\times Y^*}|\boldsymbol \phi \nabla_{\bf y}\boldsymbol \psi|\, d{\hat x}d{\bf y}.
\end{equation}
Hence, the second integral in the right-hand side of \eqref{unfoldvar} satisfies
\begin{equation}\label{pres2bis}
g\eps^2 \int_{\omega\times Y^*}\mathcal{T}_\eps(|\nabla{\bf v^\eps})|\, d{\hat x}d{\bf y} \to 0.
\end{equation}
By \eqref{Hf}, $iv)$ in Proposition \ref{properties} and  \eqref{testconv1}, we get 
\begin{equation}\label{pres3}
\eps\int_{\omega\times Y^*}\mathcal{T}_\eps ({\bf f}_\eps)\mathcal{T}_\eps({\bf v_\eps})\, d{\hat x}d{\bf y} \to 0.
\end{equation}
Finally, by \eqref{weakp} and \eqref{testconv4}, we obtain
\begin{equation}\label{pres4}
\eps\int_{\omega\times Y^*} \mathcal{T}_\eps(P_\eps|_{\Omega_\eps})\mathcal{T_\eps}\Big(\mathrm{div}_{\bf x}{\bf v}_\eps\Big)\, d{\hat x}d{\bf y}\to \int_{\omega\times Y^*} p\, \phi\, \mathrm{div}_{\bf y}\,\boldsymbol{\psi}\, d{\hat x}d{\bf y}.
\end{equation}
Then, by \eqref{pres1bis}, \eqref{pres2bis}, \eqref{pres3} and \eqref{pres4}, we can pass to the limit when $\eps$ goes to zero in \eqref{unfoldvar} and get
$$\int_{\omega\times Y^*} p\, \phi\, \mathrm{div}_{\bf y}\,\boldsymbol{\psi}\, d{\hat x}d{\bf y}\leq 0 \;\; \qquad\forall \phi \in \mathcal{D}(\omega) \text{ and  } {\boldsymbol \psi }\in H^1(Y^*)^3.$$
If we choose ${\bf u}_\eps-\eps{\bf v}_\eps$ as test function in \eqref{vpfluid}, by arguing similarly, we obtain
$$\int_{\omega\times Y^*} p\, \phi\, \mathrm{div}_{\bf y}\,\boldsymbol{\psi}\, d{\hat x}d{\bf y}\geq 0 \;\; \qquad\forall \phi \in \mathcal{D}(\omega) \text{ and  } {\boldsymbol \psi }\in H^1(Y^*)^3.$$
Thus, we can deduce
$$\int_{\omega\times Y^*} p\, \phi\, \mathrm{div}_{\bf y}\,\boldsymbol{\psi}\, d{\hat x}d{\bf y}=0 \;\; \qquad\forall \phi \in \mathcal{D}(\omega) \text{ and  } {\boldsymbol \psi }\in H^1(Y^*)^3,$$
and, by density of the tensor product $\mathcal{D}(\omega)\otimes H^1(Y^*)^3 $,
$$\int_{\omega\times Y^*} p\,\mathrm{div}_{\bf y}\,\boldsymbol{\Psi}\, d{\hat x}d{\bf y}=0 \;\; \forall \boldsymbol \Psi \in L^2(\omega; H^1(Y^*)^3),$$
which shows that the pressure $p$ doesn't depend on ${\bf y}$.
\end{proof}

\bigskip

\section{The limit problem}\label{Main}
In this section, we state and prove the main result of our paper. 
\begin{theorem}\label{teolim}
Let $({\bf u}_\eps, p_\eps)$ the unique solution of problem \eqref{vpfluid}. Let ${\bf f}_\eps \in L^2(\Omega_\eps)^3$ satisfying \eqref{Hf} and let us suppose there exists a function ${\bf f}\in L^2\big( \omega \times Y^*\big)$ such that 
\begin{equation}\label{convforce}
\mathcal{T}_\eps({\bf f}_\eps) \eto {\bf f} \quad \hbox{s}-L^2\big( \omega \times Y^*\big).
\end{equation}
Moreover let $P_\eps$ the extension of the pressure to $Q_\eps$, then there exist ${\bf u} \in L^2(\omega, H^1(Y^*)^3)$ and $p \in L^2(\omega)$ such that 
\begin{equation}\label{limfinalu}
\mathcal{T}_\eps({\bf u}_\eps)\weto  {\bf u} \quad \hbox{w}-L^2\big( \omega; H^1(Y^*)^3\big)
\end{equation}
and
\begin{equation}\label{limfinalp}
\mathcal{T}_\eps(P_\eps|_{\Omega_\eps})\weto  p \quad \hbox{w}-L^2(\omega \times Y^*),
\end{equation}
where the couple $({\bf u},\,p)$ satisfies the following limit problem
\begin{equation}\label{limitproblem}
\begin{aligned}
\mu \int_{\omega\times Y^*} & \nabla_{\bf y}{\bf u} \cdot \nabla_{\bf y}({\boldsymbol \Psi}-{\bf u})\, d{\hat x} d{\bf y} + g \int_{\omega\times Y^*}|\nabla_{\bf y}{\boldsymbol \Psi}|\, d{\hat x}d{\bf y} -g \int_{\omega\times Y^*}|\nabla_{\bf y}{\bf u}|\, d{\hat x}d{\bf y}\\
&\geqslant \int_{\omega\times Y^*} {\bf f} (\boldsymbol \Psi-{\bf u})\, d{\hat x}d{\bf y} -\int_{\omega\times Y^*} \nabla_{\hat x}p \left(\hat{\Psi}-\hat{u}\right)\, d{\hat x}d{\bf y}\quad \forall\,\boldsymbol \Psi \in \mathcal{V},
\end{aligned}
\end{equation}
with
\begin{equation*}
\begin{array}{l}
\mathcal{V}=\left\{
{\boldsymbol \Psi} \in L^2(\omega; H^1(Y^*))^3:\, {\boldsymbol \Psi}=0 \hbox{ on } \omega \times\partial Y^\ast_{inf}\cup \omega \times\partial Y^\ast_{sup},\,\mathrm{div}_{\bf y} {\boldsymbol \Psi}=0,\right.\\
\\
\left. \qquad\mathrm{div}_{\hat x}\Big(\displaystyle{\int_{Y^*} }{ \hat \Psi} d{\bf y} \Big) =0 \; \hbox{in } \omega,\, \Big(\displaystyle{\int_{Y^*} }{\hat \Psi} d{\bf y} \Big) \cdot n =0 \hbox{ on } \partial\omega\right\}.
\end{array}
\end{equation*}
\end{theorem}
\begin{proof}
In view of Proposition \ref{limit u} and Proposition \ref{limitp} there exist ${\bf u} \in L^2(\omega; H^1(Y^*)^3)$ and $p \in L^2(\omega \times Y^*)$ such that, up to a subsequence, \eqref{limfinalu} and \eqref{limfinalp} hold.

Now, we want to prove that the couple $({\bf u},\,p)$ satisfies the limit problem \eqref{limitproblem} and hence, by uniqueness, that the previous convergences hold for the whole sequences. To this aim, we consider ${\bf v}^\eps({\bf x})=\phi(\hat x) {\boldsymbol \psi}(\bf x/\eps)$ where $\phi \in \mathcal{D}(\omega)$ and ${\boldsymbol \psi }\in H^1(Y^*)^3$ is a $Y-$periodic function such that ${\boldsymbol \psi}=0 \hbox{ on } \partial Y^\ast_{inf}\cup \partial Y^\ast_{sup}$ and $\mathrm{div}_{\bf y} {\boldsymbol \psi}=0$. As previously, it is easy to show that ${\bf v}^\eps$ satisfies \eqref{gradtest}$-$\eqref{testconv3}. Moreover it holds
\begin{equation}\label{divxtest}
\mathrm{div}_{\bf x}({\bf v^\eps})=\hat\psi\Big(\frac{\bf x}{\eps}\Big) \mathrm{div}_{\hat x}\phi(\hat{x}).
\end{equation}
Hence, by $vi)$ and $vii)$ in Proposition \ref{properties}, we get
\begin{equation}\label{testconv5}
\mathcal{T_\eps}\Big(\mathrm{div}_{\bf x}{\bf v}^\eps\Big) \eto  \hat\psi \mathrm{div}_{\hat x}\phi \quad \hbox{s - }L^2(\omega \times Y^*)^3 .
\end{equation}
Let us take ${\bf v}={\bf v}^\eps$ as test function in \eqref{vpfluid}. Taking into account \eqref{pipressure} in Proposition \ref{limitp} and by applying the unfolding operator, we get
\begin{equation}\label{unfoldvarbis}
\begin{aligned}
&\mu \eps^2 \int_{\omega\times Y^*} \mathcal{T}_\eps(\nabla{\bf u}_\eps)\cdot \mathcal{T}_\eps(\nabla({\bf v}^\eps-{\bf u}_\eps))\, d{\hat x}d{\bf y} + g\eps \int_{\omega\times Y^*}\mathcal{T}_\eps(|\nabla{\bf v}_\eps)|\, d{\hat x}d{\bf y}-g\eps \int_{\omega\times Y^*}\mathcal{T}_\eps(|\nabla{\bf u}_\eps)|\, d{\hat x}d{\bf y}\\
& \geq \int_{\omega\times Y^*}\mathcal{T}_\eps ({\bf f}_\eps)\mathcal{T}_\eps({\bf v}^\eps-{\bf u}_\eps)\, d{\hat x}d{\bf y}+\int_{\omega\times Y^*} \mathcal{T}_\eps(P_\eps|_{\Omega_\eps})\mathcal{T_\eps}\Big(\mathrm{div}_{\bf x}({\bf v}^\eps-{\bf u}_\eps)\Big)\, d{\hat x}d{\bf y}.
\end{aligned}
\end{equation} 
The first integral on the left-hand side of \eqref{unfoldvarbis} can be written as

\begin{equation*}
\begin{aligned}
&\mu \eps^2 \int_{\omega\times Y^*} \mathcal{T}_\eps(\nabla{\bf u}_\eps)\cdot \mathcal{T}_\eps(\nabla({\bf v^\eps}-{\bf u}_\eps))\, d{\hat x}d{\bf y}\\
&=\mu \eps^2 \int_{\omega\times Y^*} \mathcal{T}_\eps(\nabla{\bf u}_\eps)\cdot \mathcal{T}_\eps(\nabla{\bf v^\eps})\, d{\hat x}d{\bf y}-\mu \eps^2 \int_{\omega\times Y^*} \mathcal{T}_\eps(\nabla{\bf u}_\eps)\cdot \mathcal{T}_\eps(\nabla{\bf u}_\eps)\, d{\hat x}d{\bf y}.
%\to \mu  \int_{\omega\times Y^*}\nabla_y u \cdot \big(\phi \nabla_y\psi - \nabla_yu){\bf x}d{\bf y}\\
%&=\mu \eps^2 \int_{\omega\times Y^*} \mathcal{T}_\eps(\nabla{\bf u}_\eps)\cdot \big(\mathcal{T}_\eps(\nabla_x \phi)\mathcal{T}_\eps(\psi)+ \mathcal{T}_\eps(\phi)\mathcal{T}_\eps(\nabla_y\psi\big)\, d{\bf x}d{\bf y}- \mu \eps^2 \int_{\omega\times Y^*} \mathcal{T}_\eps(\nabla{\bf u}_\eps)\cdot \mathcal{T}_\eps(\nabla({\bf u}_\eps))\, d{\bf x}d{\bf y}
\end{aligned}
\end{equation*}
As in the proof of Proposition \ref{limitp}, according to convergences \eqref{uconv2}, \eqref{testconv2}, \eqref{testconv3},  for the first term we have \eqref{pres1}.
Moreover, by standard weak lower-semicontinuity argument, we have
\begin{equation}\label{lim1}
\liminf_{\eps\to0}\eps^2 \mu\int_{\omega\times Y^*} |\mathcal{T}_\eps(\nabla{\bf u}_\eps)|^2\, d{\hat x}d{\bf y}\geq \mu\int_{\omega\times Y^*} |\nabla_{\bf y}{\bf u}|^2\, d{\hat x}d{\bf y}.
\end{equation}
As in the proof of Proposition \ref{limitp}, the second integral in the left-hand side of \eqref{unfoldvarbis} satisfies \eqref{pres2}.
\noindent Moreover, Propositon \ref{propgradv},  \eqref{uconv2} and the standard weak lower-semicontinuity argument give
\begin{equation}\label{lim3}
\liminf_{\eps\to0}g\eps \int_{\omega\times Y^*}\mathcal{T}_\eps(|\nabla{\bf u}_\eps|)\, d{\hat x}d{\bf y}\geq g \int_{\omega\times Y^*} |\nabla_{\bf y}{\bf u}|\, d{\hat x}d{\bf y}.
\end{equation}
\noindent By convergences \eqref{uconv1}, \eqref{testconv1} and \eqref{convforce} we get 
\begin{equation}\label{lim4}
\int_{\omega\times Y^*}\mathcal{T}_\eps ({\bf f}_\eps)\mathcal{T}_\eps({\bf v_\eps}-{\bf u}_\eps)\, d{\hat x}d{\bf y} \to \int_{\omega\times Y^*} {\bf f} (\phi \boldsymbol \psi-{\bf u})\, d{\hat x}d{\bf y}.
\end{equation}
Taking into account \eqref{divxtest} and that $\hbox{div}_{\bf x} {\bf u}_\eps=0$,  the second integral in the right-hand side of \eqref{unfoldvarbis} satisfies
\begin{equation*}
\begin{array}{c}
\displaystyle \int_{\omega\times Y^*} \mathcal{T}_\eps(P_\eps|_{\Omega_\eps}) \mathcal{T}_\eps(\hbox{div}_{\bf x}({\bf v_\eps}-{\bf u}_\eps))\, d{\hat x}d{\bf y}=\int_{\omega\times Y^*} \mathcal{T}_\eps(P_\eps|_{\Omega_\eps}) \mathcal{T}_\eps(\hbox{div}_{\bf x}({\bf v_\eps}))\, d{\hat x}d{\bf y}=\\
\\
=\displaystyle \int_{\omega\times Y^*} \mathcal{T}_\eps(P_\eps|_{\Omega_\eps})  \mathcal{T}_\eps(\hbox{div}_{\bf x}\phi)\hat\psi\, d{\hat x}d{\bf y}.
\end{array}
\end{equation*}
Then, by convergences \eqref{limfinalp} and \eqref{testconv5}, we obtain
\begin{equation}\label{lim5}
\int_{\omega\times Y^*} \mathcal{T}_\eps(P_\eps|_{\Omega_\eps}) \mathcal{T}_\eps(\hbox{div}({\bf v_\eps}-{\bf u}_\eps))\, d{\hat x}d{\bf y}\to \int_{\omega\times Y^*} p\, \hbox{div}_{\hat x}\phi \hat\psi\, d{\hat x}d{\bf y}.
\end{equation}
Finally, by collecting together convergences \eqref{pres1}, \eqref{pres2}, \eqref{lim1}, \eqref{lim5}, we obtain the following variational inequality
\begin{align}\label{limitd}
&\mu  \int_{\omega\times Y^*}\nabla_{\bf y} {\bf u} \cdot \big(\phi \nabla_{\bf y}\boldsymbol \psi- \nabla_{\bf y} u\big)d{\hat x}d{\bf y} + g\int_{\omega\times Y^*}|\phi \nabla_{\bf y} \boldsymbol \psi|\, d{\hat x}d{\bf y} -g \int_{\omega\times Y^*} |\nabla_{\bf y} {\bf u}|\, d{\hat x}d{\bf y}\nonumber\\ 
\nonumber\\
&\geq \int_{\omega\times Y^*} {\bf f} (\phi \boldsymbol \psi-{\bf u})\, d{\hat x}d{\bf y} + \int_{\omega\times Y^*} p \hbox{div}_{\hat x}\phi \hat \psi\, d{\hat x}d{\bf y},\\
\nonumber\\
&\forall \phi \in \mathcal{D}(\omega) \text{ and  } {\boldsymbol \psi }\in H^1(Y^*)^3 \text{ with } \hbox{div}_{\bf y}(\boldsymbol \psi)=0 \text{ and } {\boldsymbol \psi}=0 \hbox{ on } \partial Y^\ast_{inf}\cup \partial Y^\ast_{sup},\nonumber
\end{align}
which by density implies
\begin{equation}\label{limitdbis}
\begin{array}{l}
\displaystyle\mu  \int_{\omega\times Y^*}\nabla_{\bf y} {\bf u} \cdot \big(\nabla_{\bf y}\boldsymbol \Psi- \nabla_{\bf y} u\big)d{\hat x}d{\bf y} + g\int_{\omega\times Y^*}|\nabla_{\bf y} \boldsymbol \Psi|\, d{\hat x}d{\bf y} -g \int_{\omega\times Y^*} |\nabla_{\bf y} {\bf u}|\, d{\hat x}d{\bf y}\\ 
\\
\displaystyle\geq \int_{\omega\times Y^*} {\bf f} (\boldsymbol \Psi-{\bf u})\, d{\hat x}d{\bf y} + \int_{\omega\times Y^*} p \hbox{div}_{\hat x} \hat\Psi\, d{\hat x}d{\bf y},\\
\\
\forall \boldsymbol \Psi \in L^2(\omega; H^1(Y^*)^3), \hbox{with} \ \ \mathrm{div}_{\bf y} \boldsymbol \Psi=0 \text{ and }{\boldsymbol \Psi}=0 \hbox{ on } \omega \times\partial Y^\ast_{inf}\cup \omega \times\partial Y^\ast_{sup}.
\end{array}
\end{equation}
Since $p$ does not depend on $\bf y$, by \eqref{cond1} and \eqref{cond2}, we get
\begin{equation}\label{pres7}
\begin{array}{c}
\displaystyle\int_{\omega\times Y^*} p \hbox{div}_{\hat x} \hat{\Psi}\, d{\hat x}d{\bf y}=\displaystyle\int_{\omega\times Y^*} p \hbox{div}_{\hat x} \hat{\Psi}\, d{\hat x}d{\bf y}-\int_{\omega\times Y^*} p \hbox{div}_{\hat x}\left(\int_{Y^\ast}\hat{u}d{\bf y}\right)d{\hat x}=\\\\
\displaystyle\int_\omega p \hbox{div}_{\hat x}\left(\int_{Y^*} \left(\hat{\Psi}-\hat{u}\right)d{\bf y}\right)\, d{\hat x}=
-\displaystyle\int_{\omega\times Y^*} \nabla_{\hat x}p \left(\hat{\Psi}-\hat{u}\right)\, d{\hat x}d{\bf y}+\int_{\partial \omega}p\Big(\int_{Y^*} {\hat \Psi} d{\bf y} \Big) \cdot n ds.
\end{array}
\end{equation}
Therefore \eqref{limitdbis} and \eqref{pres7} imply \eqref{limitproblem}. \end{proof}

\section{Conclusions}\label{End}
In this section, we are interested in the particular interpretation of the limit problem \eqref{limitproblem} in the case of forces independent of the vertical variable, see Remark \ref{body force}.
%Notice that this assumption is natural since as the thickness tends to zero the domain shrinks in the vertical direction.
As usual in the asymptotic study of fluids in thin domains and in classical porous media, we want to describe the limit problem introducing an auxiliary problem on the basic cell. More in particular, following the ideas of Lions and Sanchez-Palencia in \cite{LS} for the study of the Bingham flow in a classical porous medium, we want to show that the limit problem \eqref{limitproblem} in Theorem \ref{teolim} can be interpreted as a non linear Darcy law. Therefore, we obtain the following proposition:
\begin{proposition}\label{limitthin}
Let ${\bf f}_\eps({{\bf x}})=(\hat f(\hat{x}),0)$ with $\hat f \in L^2(\omega)^2$. Let the velocity of filtration denoted by $$\hat V=\dfrac{1}{L_1L_2}\displaystyle\int_{Y^{\ast}} \hat u \,d{\bf y}=\int_0^{G_1}\hat{U}(\hat{x})\, dy_3,$$ 
where $\hat U$ and $\hat u$ are defined in Proposition \ref{behavior velocity} and Proposition \ref{limit u}, respectively.
Then, the limit problem \eqref{limitproblem}  is equivalent to the non linear Darcy law 
%Then, there exist ${\bf U}=(\hat{U},0) \in H^1((0,G_1);L^2(\omega)^3)$ and $p \in L^2(\omega)$ such that
%\begin{align*}
%&\tilde {\bf U }_\eps \weto  {\bf U} \quad \hbox{w}-H^1((0,G_1);L^2(\omega)^3),\\
%&P_\eps \weto p \quad \hbox{w}-L^2(\omega).
%\end{align*}
%where $\tilde {\bf U }_\eps$ is the extension by zero of ${\bf U}_\eps (\hat x, y_3)= {\bf u}_\eps(\hat x, \eps y_3)$ to the rectangle $\Omega$.
\begin{equation}\label{conditions}
\left\{
\begin{array}{ll}
\hat V(\hat{x})= \mathcal{A}(\hat f(\hat{x})-\nabla_{\hat{x}}p)& \hbox{ in } \omega,\\
\mathrm{div}_{\hat x}\hat V =0 &\hbox{ in } \omega, \\ 
\hat V(\hat{x}) \cdot n =0 &\hbox{ on }  \partial \omega,
\end{array}
\right.
\end{equation}
where $p$ is defined in Proposition \ref{limitp} and the nonlinear operator $A(\cdot):\mathbb{R}^2\to \mathbb{R}^2$ is defined by
$$\mathcal{A}(\hat{\xi})=\frac{1}{L_1L_2}\int_{Y^*} \boldsymbol \chi(\hat \xi)\,d{\bf y},$$
$\boldsymbol \chi(\hat \xi)$ being the unique solution of the following Bingham local problem on the basic cell
\begin{equation}\label{localbingham}
\left\{
\begin{array}{l}
\text{Find }\boldsymbol \chi(\hat\xi)\in \mathcal{V}\text{ such that }\\
\\
\displaystyle \mu \int_{Y^*} \nabla_{\bf y}{\boldsymbol \chi(\hat\xi)} \cdot \nabla_{\bf y}({\boldsymbol \Psi}-{\boldsymbol \chi}(\hat\xi))\,d{\bf y} + g \int_{ Y^*}|\nabla_{\bf y}{\boldsymbol \Psi}|\,d{\bf y} -g \int_{Y^*}|\nabla_{\bf y} {\boldsymbol \chi}(\hat\xi)|\,d{\bf y}
\geqslant \int_{Y^*} \hat\xi({\bf \Psi}-{\boldsymbol \chi}(\hat\xi))\,d{\bf y}\\
\\
\forall \boldsymbol \Psi \in \mathcal{V}.
\end{array}
\right.
\end{equation}
\end{proposition}
{\bf Proof.} Let us observe that under the hypotheses on the body force ${\bf f}_\eps$, \eqref{Hf} and \eqref{convforce} are satisfied and the limit problem \eqref{limitproblem} can be rewritten as
\begin{equation}\label{limitproblemD}
\begin{aligned}
\mu \int_{\omega\times Y^*} \nabla_{\bf y}{\bf u} \cdot \nabla_{\bf y}({\boldsymbol \Psi}-{\bf u})\, d{\hat x} d{\bf y} &+ g \int_{\omega\times Y^*}|\nabla_{\bf y}{\boldsymbol \Psi}|\, d{\hat x}d{\bf y} -g \int_{\omega\times Y^*}|\nabla_{\bf y}{\bf u}|\, d{\hat x}d{\bf y}\\
&\geqslant \int_{\omega\times Y^*} \left({\hat f} - \nabla_{\hat x}p\right)(\hat\Psi-{\hat u})\, d{\hat x}d{\bf y} \quad \forall\,\boldsymbol \Psi \in \mathcal{V}.
\end{aligned}
\end{equation}
For every $\hat\xi\in \mathbb{R}^2$ let $\boldsymbol \chi(\hat\xi)=\boldsymbol \chi(\bf y;\hat\xi)$ be the unique solution of problem \eqref{localbingham}.

By \eqref{localbingham} and \eqref{limitproblemD}, we get
$$
{\bf u}(\hat{x},{\bf y})={\boldsymbol \chi}({\bf y}; \hat{f}-\nabla_{\hat{x}}p).
$$
By \eqref{cond1} and \eqref{cond2}, we get
\begin{equation}\label{der1}
\left(
\int_{Y^\ast}{\hat u}(\hat{x},{\bf y})d{\bf y}, \nabla q\right)_\omega=0\qquad\forall q\in H^1(\omega).
\end{equation}
Hence, the pressure $p$ verifies
$$
\left(
\int_{Y^\ast}{\boldsymbol \chi}({\bf y}; \hat{f}-\nabla_{\hat{x}}p)d{\bf y}, \nabla q\right)_\omega=0\qquad\forall q\in H^1(\omega).
$$
Defining the nonlinear operator $\mathcal{A}:\mathbb{R}^2\to \mathbb{R}^2$ by
$$
\mathcal{A}(\hat \xi)=\dfrac{1}{L_1L_2}\int_{Y^\ast}{\boldsymbol \chi}({\bf y};\hat \xi)\, d{\bf y},
$$
the previous relation reads as
$$
\left(\mathcal{A}(\hat{f}-\nabla_{\hat{x}}p), \nabla q\right)_\omega=0\qquad\forall q\in H^1(\omega).
$$
If we define the velocity of filtration as
$$
\hat{V}(x)=\dfrac{1}{L_1L_2}\int_{Y^\ast}{\hat u}(\hat{x},{\bf y})\, d{\bf y}=\int_0^{G_1}\hat{U}(\hat{x})\, dy_3,
$$
by taking into account \eqref{cond1}, \eqref{cond2} and \eqref{der1}, we get the nonlinear Darcy's law \eqref{conditions}.$ \qquad\blacksquare$

\medskip

\begin{remark}\label{remN}
We point out that a newtonian fluid can be seen as a particular case of Bingham fluid. Thus, taking $g=0$, for any fixed $\eps$, problem \eqref{vpfluid} corresponds to the following Stokes system:
\begin{equation*}
\mu \eps^2 \int_{\Omega_\eps} \nabla{\bf u}_\eps\cdot \nabla{\bf v}\, d{\bf x}= \int_{\Omega_\eps} {\bf f}_\eps {\bf v}\, d{\bf x} +  \int_{\Omega_\eps} p_\eps {\rm div}{\bf v}\, d{\bf x}, \; \forall {\bf v} \in H^1_0(\Omega_\eps)^3,
\end{equation*}
where ${\bf f}_\eps({{\bf x}})=(f(\hat{x}),0)$ with $||f||_{L^2(\omega)^2} \leq C$.

Therefore, following a similar approach as the one used to get \eqref{limitproblem}, at the limit, we obtain the following problem
\begin{equation*}
\begin{aligned}
\mu \int_{\omega\times Y^*} \nabla_y{\bf u} \cdot \nabla_y {\boldsymbol \Psi}\, d{\hat x} d{\bf y}= \int_{\omega \times Y^* } (\hat{f}- \nabla_{\hat x} p)\hat\Psi\, d{\hat x}d{\bf y},  \quad \forall {\boldsymbol \Psi} \in \mathcal{V},
\end{aligned}
\end{equation*}
where $\mathcal{V}$ is the functional space introduced in Theorem \ref{teolim}.

%\begin{equation}
%\mu \int_{\omega} (\int_{Y^*}\nabla_y{\bf X}d{\bf y}) p \mathrm{div}_{\hat{x}}(\psi)+
%\mu \int_{\omega} (\int_{Y^*}\nabla_y{\bf X}d{\bf y}) f (\psi)=0\ \; \forall \psi \in \big(L^2(\omega; H^1_0(Y^*)^3\big), \hbox{with} \, \mathrm{div}_y \psi=0.
%\end{equation}
%
%where $X$ is the solution of 
%\begin{equation*}
%\begin{aligned}
%&\mu \int_{Y^*} \nabla_y{\bf X} \cdot \nabla_y({\bf \psi})\,d{\bf y} 
%= \int_{Y^*}({\bf \psi}\,d{\bf y} \; \forall \psi \in ; H^1_0(Y^*)^3, \hbox{with} \, \mathrm{div}_y \psi=0\\
%&\mathrm{div}_{y}{\bf X}=0 \hbox{ in } Y^*,\\
%&{\bf X}=0 \hbox{ on } \{y_3=0\}\cup \omega \times \{y_3=G(\hat{y})\}.
%\end{aligned} 
%\end{equation*}

In fact, notice that we can recover the convergence results given in \cite{BC}, see also \cite{FKT,FTP} for a generalization to the unstationary case, without using extension operators. For instance, Theorem 3.1 and Theorem 3.2 are equivalent to our Proposition \ref{behavior velocity} and Proposition \ref{limitp}, respectively. 
\end{remark}

\bigskip

{\bf Acknowledgments.} This paper was initiated during the visit of the last author at the University of Sannio, whose warm hospitality and support are gratefully acknowledged. G.C. and C.P. are members of GNAMPA (INDAM). The last author was partially supported by grant MTM2016-75465-P from the Ministerio de Economia y Competitividad, Spain and Grupo de Investigación CADEDIF, UCM.

\end{document}